\documentclass{amsart}
\usepackage{amssymb,amstext,amsmath,amscd,amsthm,amsfonts,enumerate,graphicx,latexsym,stmaryrd,multicol, mathrsfs}

\usepackage[usenames]{color}
\usepackage[all]{xy}
\usepackage{extarrows}
\usepackage{mathabx}
\usepackage{cite}
\usepackage{textcomp}

\usepackage[top=25truemm,bottom=25truemm,left=25truemm,right=25truemm]{geometry}

\newtheorem{thm}{Theorem}[section]
\newtheorem{lem}[thm]{Lemma}
\newtheorem{prop}[thm]{Proposition}
\newtheorem{cor}[thm]{Corollary}
\theoremstyle{definition}
\newtheorem{dfn}[thm]{Definition}

\newtheorem{rmk}[thm]{Remark}
\theoremstyle{remark}

\newtheorem*{ac}{Acknowledgments}

\numberwithin{equation}{thm}

\def\Cok{\operatorname{Coker}}
\def\Cokv{\operatorname{\underline{Coker}}}

\def\dell{\operatorname{dell}}
\def\depth{\operatorname{depth}}
\def\Ext{\operatorname{Ext}}

\def\ge{\geqslant}

\def\grade{\operatorname{grade}}

\def\H{\mathrm{H}}

\def\Hom{\operatorname{Hom}}

\def\image{\operatorname{Im}}

\def\Ker{\operatorname{Ker}}

\def\le{\leqslant}
\def\lhom{\operatorname{\underline{Hom}}}
\def\lmod{\operatorname{\underline{mod}}}
\def\m{\mathfrak{m}}

\def\mod{\operatorname{mod}}

\def\p{\mathfrak{p}}
\def\pd{\operatorname{pd}}

\def\proj{\operatorname{proj}}

\def\syz{\Omega}

\def\tr{\operatorname{Tr}}
\def\V{\mathrm{V}}
\def\X{\mathcal{X}}

\begin{document}
\allowdisplaybreaks
\title[An approach to Martsinkovsky\textquotesingle s invariant via Auslander\textquotesingle s approximation theory]{An approach to Martsinkovsky\textquotesingle s invariant\\ via Auslander\textquotesingle s approximation theory}
\author{Yuya Otake}
\address{Graduate School of Mathematics, Nagoya University, Furocho, Chikusaku, Nagoya 464-8602, Japan}
\email{m21012v@math.nagoya-u.ac.jp}

\thanks{2020 {\em Mathematics Subject Classification.} 13D02, 16D90, 13D07.}
\thanks{{\em Key words and phrases.} $\xi$-invariant, $\delta$-invariant, approximation theory, grade, stable category, syzygy.}
\thanks{The author was partly supported by Grant-in-Aid for JSPS Fellows 23KJ1119.}
\begin{abstract}
Auslander developed a theory of the $\delta$-invariant for finitely generated modules over commutative Gorenstein local rings, and Martsinkovsky extended this theory to the $\xi$-invariant for finitely generated modules over general commutative noetherian local rings.
In this paper, we approach Martsinkovsky’s $\xi$-invariant by considering a non-decreasing sequence of integers that converges to it.
We investigate Auslander’s approximation theory and provide methods for computing this non-decreasing sequence using the approximation.
\end{abstract}
\maketitle
\tableofcontents
\section{Introduction}
The theory of maximal Cohen--Macaulay  (abbreviated to MCM) approximations was established by Auslander and Buchweitz \cite{ABu}, and has played an important role in Cohen--Macaulay representation theory.  
Moreover, Auslander showed that every finitely generated module $M$ over a commutative Gorenstein local ring admits a unique minimal MCM approximation in his unpublished paper \cite{AusM}, and defined the {\em $\delta$-invariant} of $M$, denoted $\delta(M)$, using the approximation.
Numerous interesting properties and applications of the delta invariant have been discovered; see, for example, \cite{ADS, ABIM, Ding92, Ding93, Ding94, Ding942, HS, Kfree, Puthen, Shida, Yoshida, Yos96, Yos99},  
and also the detailed expositions in \cite{LW, Yos93}.
Dual to the existence of MCM approximations, every finitely generated module $M$ over a Gorenstein local ring admits a finite projective dimension (abbreviated to FPD) hull $0 \to M \to Y \to X \to 0$.
Auslander \cite{AusM} showed that the delta invariant can be computed via an FPD hull, and as a consequence, he found striking properties of modules with null $\delta$-invariant.
Here, we denote by $\mu(M)$ the minimal number of generators of an $R$-module $M$.

\begin{thm}[Auslander]\label{IntroAus}
Let $R$ be a commutative Gorenstein local ring with residue field $k$, and let $M$ be a finitely generated $R$-module.
\begin{enumerate}[\rm(1)]
   \item
   For any FPD hull $0 \to M \to Y \to X \to 0$ of $M$, the equality $\delta(M) = \mu(Y) - \mu(X)$ holds.
   \item
   The following conditions are equivalent.
   \begin{itemize}
   \item[(i)]
   The equality $\delta(M) = 0$ holds.
   \item[(ii)]
   For every $R$-homomorphism $f : M \to Z$ with $Z$ a finitely generated $R$-module of finite projective dimension, $f \otimes_R k = 0$.
   \end{itemize}
\end{enumerate}
\end{thm}

Let us now move on to Auslander’s approximation theory over more general two-sided noetherian rings.
Auslander and Bridger \cite{AB} developed an approximation theory for modules satisfying certain grade conditions, which refines the theory of MCM approximations over commutative Gorenstein local rings.  
Simon \cite{Sim} studied the dual notion of Auslander--Bridger approximations, called the {\em $n$-hull}.  
Over commutative Gorenstein local rings, Kato \cite{Kview} constructed a theory of {\em origin extensions}, which lie between MCM approximations and FPD hulls.
Here, over a general two-sided noetherian ring $\Lambda$, we formulate these approximations, extensions, and hulls as follows.
We denote by $\mod\Lambda$ the category of finitely generated right $\Lambda$-modules, and by $\pd_\Lambda M$ the projective dimension of a $\Lambda$-module $M$.

\begin{dfn}
Let $n \ge 0$ be an integer. We define the full subcategories $\mathscr{A}_n(\Lambda)$, $\mathscr{E}_n(\Lambda)$, and $\mathscr{H}_n(\Lambda)$ of $\mod \Lambda$ as follows:
\[
\begin{aligned}
\mathscr{A}_n(\Lambda) &= \left\{ 
  M \in \mod\Lambda \,\middle|\,
  \begin{matrix}\begin{aligned}&\text{there is an exact sequence } 0 \to Y \to X \to M \to 0 \text{ in } \mod\Lambda \text{ such that} \\
    &\pd_\Lambda Y\le  n - 1 \text{ and } \Ext^i_\Lambda(X,\Lambda) = 0 \text{ for all } 1 \le i \le n
  \end{aligned}\end{matrix}
\right\}, \\
\mathscr{E}_n(\Lambda) &= \left\{ 
  M \in \mod\Lambda \,\middle|\,
  \begin{matrix}\begin{aligned}
  &\text{there is an exact sequence } 0 \to X \to M \oplus P \to Y \to 0 \text{ in } \mod\Lambda \text{ such that}  \\
    &P \text{ is projective, } \pd_\Lambda Y\le n,  \text{ and } \Ext^i_\Lambda(X,\Lambda) = 0 \text{ for all } 1 \le i \le n
   \end{aligned}\end{matrix}
\right\}, \\
\mathscr{H}_n(\Lambda) &= \left\{ 
  M \in \mod\Lambda \,\middle|\,
  \begin{matrix}\begin{aligned}
  &\text{there is an exact sequence } 0 \to M \to Y \to X \to 0 \text{ in } \mod\Lambda \text{ such that} \\
    &\pd_\Lambda Y \le n \text{ and } \Ext^i_\Lambda(X,\Lambda) = 0 \text{ for all } 1 \le i \le n+1
  \end{aligned}\end{matrix}
\right\}.
\end{aligned}
\]

The short exact sequences appearing in the definitions of $\mathscr{A}_n(\Lambda)$, $\mathscr{E}_n(\Lambda)$, and $\mathscr{H}_n(\Lambda)$ are called an {\em $n$-AB approximation}, {\em $n$-origin extension}, and {\em $n$-FPD hull} of $M$, respectively.
\end{dfn}

Under the above definition, we have the inclusions $\mathscr{A}_n(\Lambda) \supset \mathscr{E}_n(\Lambda) \supset \mathscr{H}_n(\Lambda)$.
The subcategories $\mathscr{A}_n(\Lambda)$ and $\mathscr{E}_n(\Lambda)$ have been essentially studied in the Auslander--Bridger theory \cite{AB}.  
In particular, they proved that any finitely generated module satisfying the {\em $n$-th grade condition} (see Definition~\ref{nthgradecond}) admits an $n$-origin extension.  
More precisely, they used the $n$-th grade condition to construct a well-behaved filtration, called an {\em $n$-spherical filtration}, and used it to prove the result.
In Section~3 of the present paper, we prove the converse: if a finitely generated module admits an $n$-spherical filtration, then it satisfies the $n$-th grade condition. 
Furthermore, we show that if $\Lambda$ satisfies a certain grade condition on modules of finite projective dimension, then any finitely generated $\Lambda$-module admitting an $n$-origin extension also satisfies the $n$-th grade condition.
Specifically, when $\Lambda$ is commutative, the following three conditions are equivalent: satisfying the $n$-th grade condition, admitting an $n$-spherical filtration, and admitting an $n$-origin extension; see Theorem~\ref{norgeq}.
Also, the category $\mathscr{H}_n(\Lambda)$ is related to the notion of {\em delooping level}, which was introduced by G\'{e}linas \cite{Gel} and has been actively studied in recent years.  
In fact, as a consequence of Theorem~\ref{nFPDhullex}, any object of $\mathscr{H}_n(\Lambda)$ has delooping level at most $n$; see Remark~\ref{dell}(2).

As a generalization of Auslander's $\delta$-invariant, Martsinkovsky \cite{Mart96I} introduced the {\em $\xi$-invariant}, which is denoted by $\xi(M)$, for a finitely generated module $M$ over a commutative noetherian local ring $R$.  
He showed that the following theorem.

\begin{thm}[Martsinkovsky]\label{IntroMart}
Let $R$ be a commutative Gorenstein local ring, and let $M$ be a finitely generated $R$-module.
Then the equality $\delta(M)=\xi(M)$ holds.
\end{thm}

Various properties that were known for the $\delta$-invariant have also been established for the $\xi$-invariant; see \cite{Mart96I, Mart96rem, Mart96II}.
However, the $\xi$-invariant is generally difficult to compute, as its definition involves Tate--Vogel cohomology.  
Our strategy, therefore, is to instead treat the following invariant. 
Let $R$ be a commutative noetherian local ring with residue field $k$, and let $M$ be a finitely generated $R$-module and $n \ge 0$ an integer.  
We define the {\em $n$-th approximated $\xi$-invariant} $\xi_n(M)$ by  
$$
\xi_n(M) := \dim_k \Ker\big( \Hom_R(M,k) \xrightarrow{\pi} \lhom_R(M,k) \xrightarrow{\syz^n} \lhom_R(\syz^n M,\syz^n k) \big),
$$
where $\lhom_R(-,-)$ denotes the quotient of $\Hom_R(-,-)$ by the submodule consisting of homomorphisms factoring through projective modules,  
$\pi$ is the canonical surjection, and $\syz^n$ denotes the $n$-th syzygy functor.
Then, we obtain a non-decreasing sequence of integers
$$
\xi_0(M) \le \xi_1(M) \le \cdots \le \xi_n(M) \le \xi_{n+1}(M) \le \cdots \le \xi(M),
$$
and the equality $$\xi(M) = \lim_{n \to \infty} \xi_n(M)$$ holds.
Our main results give expressions for the $n$-th approximated $\xi$-invariant of modules belonging to the subcategories $\mathscr{A}_n(R)$, $\mathscr{E}_n(R)$, or $\mathscr{H}_n(R)$.

\begin{thm}[Theorems~\ref{ABrank}, \ref{orginm} and Proposition~\ref{FPDxi}]\label{Main}
Let $R$ be a commutative noetherian local ring with residue field $k$.  
Let $M$ be a finitely generated $R$-module, and let $n \ge 0$ be an integer.
\begin{enumerate}[\rm(1)]
   \item
   Suppose that $M$ belongs to $\mathscr{A}_n(R)$.  
   Take a minimal $n$-AB approximation $0 \to Y_M \to X_M \to M \to 0$ of $M$.  
   Then $\xi_n(M)$ coincides with the rank of the maximal free summand of $X_M$.
   \item
   Suppose that $M$ belongs to $\mathscr{E}_n(R)$.  
   Then the following conditions are equivalent:
   \begin{itemize}
      \item[(i)]
      The equality $\xi_n(M) = 0$ holds.
      \item[(ii)]
      For every $R$-homomorphism $f : M \to Z$, where $Z$ is a finitely generated $R$-module of projective dimension at most $n$, $f \otimes_R k = 0$.
   \end{itemize}
   \item
   Suppose that $M$ belongs to $\mathscr{H}_n(R)$.
   For any $n$-FPD hull $0 \to M \to Y \to X \to 0$ of $M$, the equality $\xi_n(M) = \mu(Y) - \mu(X)$ holds.
\end{enumerate}
\end{thm}

The minimality of the $n$-AB approximation appearing above is defined in the same way as in the case of MCM approximations; see Definition~\ref{mindef}.
We should mention that our main Theorem~\ref{Main} recovers both Theorem~\ref{IntroAus} due to Auslander and Theorem~\ref{IntroMart} due to Martsinkovsky; it is also worth noting that our methods give a much more elementary proof to Martsinkovsky's theorem, whose original proof is based on differential graded structures of free resolutions.
We remark that the proof of Theorem~\ref{IntroAus}(2) usually relies on the existence of an FPD hull, whereas our approach requires only the weaker assumption that $M$ belongs to $\mathscr{E}_n(R)$.

The organization of this paper is as follows.
In Section 2, we state several notions and their basic properties for later use.
In particular, we recall some fundamental concepts from stable module theory and the notion of approximations in the sense of Auslander--Smal\o.
In Section~3, for a general two-sided noetherian ring $\Lambda$, we describe the behavior of the subcategories $\mathscr{A}_n(\Lambda)$, $\mathscr{E}_n(\Lambda)$, and $\mathscr{H}_n(\Lambda)$ and their relation to Auslander’s grade theory.
Some of the results in this section are reformulations of results from \cite{AB, Omorp, OExt}, but we give explicit constructions of $n$-AB approximations, $n$-origin extensions, and $n$-FPD hulls.
We also prove the converse of a theorem of Auslander--Bridger concerning the existence of $n$-spherical filtrations mentioned above.
In Section 4, we study the $n$-th approximated $\xi$-invariant over commutative noetherian local rings.
The approach taken by Martsinkovsky in \cite{Mart96I, Mart96rem, Mart96II} to investigate the $\xi$-invariant relies on the theory of complexes, whereas our approach is based on the theory of stable module categories.
We not only prove Theorem \ref{Main}, but also investigate when the sequence of approximated $\xi$-invariants stabilizes, and how it behaves modulo a regular element.
As a consequence, we refine various results of Martsinkovsky; see Corollary~\ref{cors}.

\section{Preliminaries}
Throughout this paper, assume that all rings are two-sided noetherian, and call right modules simply modules.  
Let $\Lambda$ be a two-sided noetherian ring.  
We denote by $\mod\Lambda$ the category of finitely generated (right) $\Lambda$-modules, and by $\proj\Lambda$ the category of finitely generated projective $\Lambda$-modules.  
The opposite ring of $\Lambda$ is denoted by $\Lambda^{\mathrm{op}}$, and the $\Lambda$-dual $\Hom_\Lambda(-,\Lambda)$ is denoted by $(-)^\ast$.

\begin{dfn}
We denote by $\lmod\Lambda$ the {\em stable category} of $\mod\Lambda$.  
The objects of $\mod\Lambda$ are the same as those of $\mod\Lambda$.  
For objects $X, Y$ in $\mod\Lambda$, the morphism set is defined by
\[
\Hom_{\mod\Lambda}(X, Y) = \lhom_\Lambda(X, Y) = \Hom_\Lambda(X, Y)/\mathscr{P}(X, Y),
\]
where $\mathscr{P}(X, Y)$ denotes the subgroup of $\Hom_\Lambda(X, Y)$ consisting of $\Lambda$-homomorphisms that factor through some projective module.  
For any $\Lambda$-homomorphism $f : X \to Y$, we denote by $\underline{f}$ the image of $f$ in $\lhom_\Lambda(X, Y)$.
\end{dfn}

The following syzygy functor and (Auslander) transpose functor play a central role in the theory of stable module categories.

\begin{dfn}
Let $M$ be a finitely generated $\Lambda$-module and take a (finite) projective presentation $P_1 \xrightarrow{\partial} P_0 \to M \to 0$.
\begin{enumerate}[\rm(1)]
\item
We write the image of $\partial$ as $\syz M$ and call it the {\em syzygy} of $M$.
Then $\syz M$ is uniquely determined by $M$ up to projective summands. Taking the syzygy induces an additive functor $\syz : \lmod\Lambda \to \lmod\Lambda$.
For each integer $n \ge 1$, we define $\syz^n$ inductively by $\syz^n = \syz \circ \syz^{n-1}$.
\item
We write the cokernel of the $\Lambda$-dual $\partial^\ast$ as $\tr M$ and call it the {\em (Auslander) transpose} of $M$.
Then $\tr M$ is uniquely determined by $M$ up to projective summands. Taking the transpose induces an anti-equivalence $\tr : \lmod\Lambda \to \lmod\Lambda^{\mathrm{op}}$.
\end{enumerate}
\end{dfn}

It is well known that the pair of endofunctors $(\tr\syz\tr, \syz)$ on the stable category $\lmod\Lambda$ forms an adjoint pair; see \cite{AB} or \cite[Section~4]{Omorp} for instance.
Similarly, for any integer $n \ge 1$, the pair $(\tr\syz^n\tr, \syz^n)$ also forms an adjoint pair; that is, there are functorial isomorphisms
$$
\theta_{M, N}^n : \lhom_\Lambda(M, \syz^n N) \xrightarrow{\sim}\lhom_\Lambda(\tr\syz^n\tr M, N)$$
for all $M$, $N \in \lmod\Lambda$.
We denote the counit morphism of the above adjoint pair by $\underline{\psi^n_{(-)}} : \tr\syz^n\tr\syz^n(-) \to (-)$.  
That is, for each $M \in \lmod\Lambda$, the morphism $\underline{\psi^n_M}$ in $\lmod\Lambda$ is given by  
$$
\underline{\psi^n_M} = \theta^n_{\syz^n M, M}(\underline{1_{\syz^n M}}) : \tr\syz^n\tr\syz^n M \to M.
$$

Here, we recall the notions of right and left approximations by subcategories, as well as right and left minimality of morphisms, in the sense of Auslander and Smal\o\ \cite{AS}.

\begin{dfn}
Let $\mathscr{A}$ be an additive category, and let $\X$ be a subcategory of $\mathscr{A}$.
\begin{enumerate}[\rm(1)]
\item
A morphism $f : X \to M$ in $\mathscr{A}$ is called a {\em right $\X$-approximation} of $M$ if $X$ is an object of $\X$, and for any morphism $f' : X' \to M$ with $X' \in \X$, there exists a morphism $p : X' \to X$ such that $f' = f \circ p$. In other words, the map $\Hom_{\mathscr{A}}(X', f) : \Hom_{\mathscr{A}}(X', X) \to \Hom_{\mathscr{A}}(X', M)$ is surjective for any $X' \in \X$.
\item
A morphism $g : M \to X$ in $\mathscr{A}$ is called a {\em left $\X$-approximation} of $M$ if $X$ is an object of $\X$, and for any morphism $g' : M \to X'$ with $X' \in \X$, there exists a morphism $q : X \to X'$ such that $g' = q \circ g$. In other words, the map $\Hom_{\mathscr{A}}(g, X') : \Hom_{\mathscr{A}}(X, X') \to \Hom_{\mathscr{A}}(M, X')$ is surjective for any $X' \in \X$.
\end{enumerate}
\end{dfn}

\begin{dfn}
Let $\mathscr{A}$ be an additive category, and let $f : X \to Y$ be a morphism in $\mathscr{A}$.
\begin{enumerate}[\rm(1)]
\item
The morphism $f$ is said to be {\em right minimal} if every morphism $p : X \to X$ satisfying $f \circ p = f$ is an automorphism.
\item
The morphism $f$ is said to be {\em left minimal} if every morphism $q : Y \to Y$ satisfying $q \circ f = f$ is an automorphism.
\end{enumerate}
\end{dfn}

We will always deal with approximations within the category $\mod\Lambda$ of finitely generated $\Lambda$-modules.  
For morphisms $f : P \to M$ and $g : M \to Q$ in $\mod\Lambda$, where $P, Q \in \proj\Lambda$, note that $f$ is a right $\proj\Lambda$-approximation if and only if $f$ is surjective, and $g$ is a left $\proj\Lambda$- approximation if and only if $g^\ast$ is surjective.

The notion of grade, introduced in the 1950s by D. Rees in commutative ring theory, plays a central role in this paper as well.  
The theory of grade over noncommutative rings has also been extensively studied; see, for example, \cite{AusC, AB, AR1, AR2, HI}.

\begin{dfn}
Let $M$ be a $\Lambda$-module. The {\em grade} of $M$, written as $\grade_\Lambda M$, is defined to be the smallest integer $i \ge 0$ such that $\Ext^i_\Lambda(M, \Lambda) \ne 0$.
\end{dfn}
\section{Auslander's grade theory and approximation theory}

In this section, we consider approximation theory over general two-sided noetherian rings.  
The results presented in this section will play an important role in the next section.

Throughout this section, let $\Lambda$ be a two-sided noetherian ring.
The theory of $n$-AB approximations described below was developed by Auslander and Bridger \cite{AB}.
Here, “AB” stands for both Auslander--Bridger and Auslander--Buchweitz.  
The existence of $n$-origin extensions, named by Kato \cite{Kview}, was also established by Auslander and Bridger.

\begin{dfn}
Let $\Lambda$ be a two-sided noetherian ring, and let $M$ be a finitely generated $\Lambda$-module.
Let $n\ge0$ be an integer.
\begin{enumerate}[\rm(1)]
    \item
    A short exact sequence $0 \to Y \to X \to M \to 0$ in $\mod\Lambda$ is called an {\em $n$-AB approximation} of $M$ if $Y$ has projective dimension at most $n-1$ and $\Ext^i_{\Lambda}(X,\Lambda)=0$ for all $1\le i\le n$.
    We denote by $\mathscr{A}_n(\Lambda)$ the subcategory of $\mod\Lambda$ consisting of modules having $n$-AB approximations.
    \item 
    A short exact sequence $0 \to X \to M\oplus P \to Y \to 0$ in $\mod\Lambda$ is called an {\em $n$-origin extension} of $M$ if $P$ is projective, $Y$ has projective dimension at most $n$ and $\Ext^i_{\Lambda}(X,\Lambda)=0$ for all $1\le i\le n$.
    We denote by $\mathscr{E}_n(\Lambda)$ the subcategory of $\mod\Lambda$ consisting of modules having $n$-origin extensions.
    \item
    A short exact sequence $0 \to M\to Y\to X \to 0$ in $\mod\Lambda$ is called an {\em $n$-finite projective dimension hull} (simply, {\em $n$-FPD hull}) of $M$ if $Y$ has projective dimension at most $n$ and $\Ext^i_{\Lambda}(X,\Lambda)=0$ for all $1\le i\le n+1$.
    We denote by $\mathscr{H}_n(\Lambda)$ the subcategory of $\mod\Lambda$ consisting of modules having $n$-fulls.
\end{enumerate}
\end{dfn}

Let us record some immediate consequences from the definitions of the three subcategories above.

\begin{rmk}\label{bydef}
Let $n\ge0$ be an integer.
Denote by $\mathscr{T}_n(\Lambda)$ the subcategory of $\mod\Lambda$ consisting of modules $M$ satisfying $\Ext^i(M,\Lambda) = 0$ for all $1 \le i \le n$.  
Similarly, denote by $\mathscr{P}_n(\Lambda)$ the subcategory of $\mod\Lambda$ consisting of modules $M$ with projective dimension at most $n$.
\begin{enumerate}
   \item
   By definition, we have $\mathscr{A}_0(\Lambda) = \mathscr{E}_0(\Lambda) = \mod\Lambda$.  
   Moreover, as will be stated later, a consequence of the approximation theory developed by Auslander and Bridger shows that $\mathscr{A}_1(\Lambda) = \mod\Lambda$.  
   The subcategory $\mathscr{H}_0(\Lambda)$ is nothing but the subcategory of $\mod\Lambda$ consisting of finitely generated torsionless $\Lambda$-modules.
   Here, for the natural homomorphism $\varphi_M : M \to M^{\ast\ast}$ given by $\varphi_M(x)(f) = f(x)$ for $x \in M$ and $f \in M^\ast$, the module $M$ is said to be {\em torsionless} if $\varphi_M$ is injective, and {\em reflexive} if $\varphi_M$ is an isomorphism.
   \item
   Suppose that a finitely generated $\Lambda$-module $M$ has an $n$-AB approximation $0\to Y\to X\xrightarrow{f}M\to0$.
   Then $f$ is a right $\mathscr{T}_n(\Lambda)$-approximation of $M$.
   Indeed, note that $\Ext^1(X,Y)=0$ for all $X\in\mathscr{T}_n(\Lambda)$ and $Y\in\mathscr{P}_{n-1}(\Lambda)$.
   Therefore, for any morphism $f':X'\to M$ in $\mod\Lambda$ with $X'\in\mathscr{T}_n(\Lambda)$, we have an exact sequence
   $$
   \Hom(X',X)\to\Hom(X',M)\to\Ext^1(X',Y)=0,
   $$
   and so $f'$ factors through $f$.
   Dually, if $M$ admits an $n$-FPD hull $0 \to M \xrightarrow{g} Y \to X \to 0$, then $g$ is a left $\mathscr{P}_n(\Lambda)$-approximation.
   \item 
   There are inclusions $\mathscr{A}_n(\Lambda)\supset\mathscr{E}_n(\Lambda)\supset\mathscr{H}_n(\Lambda)$.
   For example, suppose that a finitely generated $\Lambda$-module $M$ has an $n$-FPD hull $0\to M\to Y\to X\to 0$.
   Taking the syzygy of $X$, we have an exact sequence $0\to \syz X\to M\oplus P\to Y\to0$ with $P\in\proj\Lambda$.
   This sequence is an $n$-origin hull of $M$.
   The other inclusion is proved in a similar way.
\end{enumerate}
\end{rmk}

The grade of Ext modules has been actively studied, mainly in the representation theory of non-commutative rings; see, for example, \cite{AB, AR1, AR2, Gel, H, HI, Omorp}.  
The behavior of the subcategories $\mathscr{A}_n(\Lambda)$, $\mathscr{E}_n(\Lambda)$, and $\mathscr{H}_n(\Lambda)$ defined above is closely related to grade theory.  
The terminology in (1) below is due to G\'{e}linas \cite{Gel}.

\begin{dfn}\label{nthgradecond}
Let $n \ge 0$ be an integer.
\begin{enumerate}[\rm(1)]
  \item
  Let $M$ be a finitely generated $\Lambda$-module.  
  We say that $M$ satisfies the {\em $n$-th grade condition} if the inequality
  $
  \grade_{\Lambda^{\mathrm{op}}} \Ext^i_\Lambda(M, \Lambda) \ge i
  $
  holds for every $1 \le i \le n$.
  \item
  We say that the ring $\Lambda$ satisfies the {\em $(\ast_n)$-condition} if every finitely generated $\Lambda$-module $Z$ with projective dimension at most $n$ satisfies the $n$-th grade condition.
\end{enumerate}
\end{dfn}

The condition in (2) was considered in \cite[Remark on page 70]{AB}, and serves as a key connection between approximation theory and grade theory.  
Typical examples of modules satisfying the $n$-th grade condition are $n$-spherical modules.  
For an integer $n > 0$, a finitely generated $\Lambda$-module $M$ is said to be {\em $n$-spherical} if it has projective dimension at most $n$ and satisfies $\Ext^i_\Lambda(M, \Lambda) = 0$ for all $1 \le i \le n-1$.
For further details on spherical modules, see also \cite{AB, Gel, Ost}.

\begin{lem}\label{astn}
Let $n > 0$ be an integer.  
Then the following statements hold.
\begin{enumerate}[\rm(1)]
  \item
  Every $n$-spherical module satisfies the $n$-th grade condition.
  \item
  Every noetherian ring satisfies the condition $(\ast_1)$.
  \item
  Every commutative noetherian ring satisfies the condition $(\ast_m)$ for all integers $m > 0$.
\end{enumerate}
\end{lem}

\begin{proof}
The statement (1) is a consequence of \cite[Theorem 1.1]{Ost}.  
By setting $n = 1$ in (1), we obtain (2).  
The statement (3) is nothing but \cite[Corollary 4.18]{AB}.
\end{proof}

The notion of $n$-torsionfree modules plays an important role in the stable module theory developed by Auslander and Bridger.  
Let $n \ge 0$ be an integer.  
A finitely generated $\Lambda$-module $M$ is said to be {\em $n$-torsionfree} if $\Ext^i_{\Lambda^{\mathrm{op}}}(\tr M, \Lambda) = 0$ for all $1 \le i \le n$. 
For a finitely generated $\Lambda$-module $M$, we have the Auslander exact sequence \cite{AB, AusC}:
$$
0\to \Ext^1_{\Lambda^{\mathrm{op}}}(\tr M,\Lambda)\to M\xrightarrow{\varphi_M}M^{\ast\ast}\to\Ext^2_{\Lambda^{\mathrm{op}}}(\tr M,\Lambda)\to 0.
$$
This exact sequence implies that 1-torsionfreeness is equivalent to being torsionless, and 2-torsionfreeness is equivalent to being reflexive.
The following theorem is a consequence of Auslander--Bridger theory and provides a fundamental result on the existence of $n$-AB approximations.

\begin{thm}[Auslander--Bridger]\label{thmnABapp}
Let $M$ be a finitely generated $\Lambda$-module and $n>0$ an integer.
Consider the following conditions.
\begin{enumerate}[\rm(1)]
   \item
   One has the inequality $\grade_{\Lambda^{\mathrm{op}}}\Ext^i_\Lambda(M,\Lambda)\ge i-1$ for all $1\le i\le n-1$.
   \item
   The $i$-th syzygy $\syz^i M$ is $i$-torsionfree for all $1\le i\le n$.
   \item
   The $n$-th syzygy $\syz^n M$ is $n$-torsionfree.
   \item
   The module $M$ is an object of $\mathscr{A}_n(\Lambda)$, i.e., $M$ has an $n$-AB approximation.
\end{enumerate}
Then the implications $(1)\Longleftrightarrow(2)\Longrightarrow (3)\Longleftrightarrow (4)$ hold.
If the equivalent conditions (3) and (4) are satisfied, then for a right $\proj \Lambda$-approximation $s: P \to M$ of $M$, the short exact sequence
$$
0 \to Y \to \tr \syz^n \tr \syz^n M \oplus P \xrightarrow{(\psi^n_M,\, s)} M \to 0,
$$
where $Y = \Ker(\psi^n_M, s)$, gives an $n$-AB approximation of $M$.
Also, if $\Lambda$ satisfies the $(\ast_{n-1})$-condition, then all conditions are equivalent.
In particular, when $\Lambda$ is commutative, all conditions are equivalent.
\end{thm}

\begin{proof}
The equivalence $(1) \Longleftrightarrow (2)$ is nothing but \cite[Proposition 2.26]{AB}.  
The implication $(2) \Longrightarrow (3)$ is clearly true.  
The equivalence $(3) \Longleftrightarrow (4)$ is also nothing but \cite[Proposition 2.21]{AB}, whose proof shows that the morphism $\underline{\psi^n_M} : \tr \syz^n \tr \syz^n M \to M$ gives an $n$-AB approximation.  
If $\Lambda$ satisfies the $(\ast_{n-1})$ condition, then the implication $(3) \Longrightarrow (2)$ follows from \cite[Corollary 2.27]{AB}.
\end{proof}

Next, we describe the relationship between $n$-origin extensions and the $n$-th grade condition.  
To do this, the notion of representation by monomorphisms plays an important role.
A $\Lambda$-homomorphism $f : X \to Y$ of finitely generated $\Lambda$-modules is said to be {\em represented by monomorphisms} if there exist $P, Q \in \proj \Lambda$ and $\Lambda$-homomorphisms $s, t, u$ such that the homomorphism
$$
\begin{pmatrix} f & s \\ t & u \end{pmatrix} : X \oplus P \to Y \oplus Q
$$
is a monomorphism.
It is easy to see that this condition is equivalent to the existence of a projective module $Q \in \proj \Lambda$ and a $\Lambda$-homomorphism $t : X \to Q$ such that the map $\binom{f}{t} : X \to Y \oplus Q$ is a monomorphism.  
For the theory of representation by monomorphisms, we refer the reader to \cite{AB, Kmorp, Omorp}.  
Under this notion, the following theorem holds.  
Auslander and Bridger proved that the $n$-th grade condition implies the existence of $n$-origin extensions and $n$-spherical approximations.  
We will prove the converse of Auslander and Bridger’s theorem.

\begin{thm}\label{norgeq}
Let $M$ be a finitely generated $\Lambda$-module and $n \ge 0$ an integer.  
Consider the following conditions.
\begin{enumerate}[\rm(1)]
   \item
   One has the inequality $\grade_{\Lambda^{\mathrm{op}}} \Ext^i_\Lambda(M, \Lambda) \ge i$ for all $1 \le i \le n$.
   \item
   One has the inequality $\grade_{\Lambda^{\mathrm{op}}} \Ext^i_\Lambda(M, \Lambda) \ge i - 1$, and the morphism $\underline{\psi^i_M} : \tr \syz^i \tr \syz^i M \to M$ is represented by monomorphisms for all $1 \le i \le n$.
   \item
   The module $M$ has an $n$-th spherical approximation, i.e., there exists a filtration
   $$
   M_n \subset M_{n-1} \subset \cdots \subset M_1 \subset M_0 = M \oplus P
   $$
   with $P \in \proj \Lambda$, such that for every $1 \le k \le n$, the following conditions hold:  
   $\Ext^i_{\Lambda}(M_k, \Lambda) = 0$ for all $1 \le i \le k$,  
   the quotient $M_{k-1} / M_k$ is $k$-spherical, and the $\Lambda$-dual map $M_{k-1}^\ast \to M_k^\ast$ is surjective.
   \item
   The module $M$ is an object of $\mathscr{E}_n(\Lambda)$, i.e., $M$ has an $n$-origin extension.
   \item
   The module $M$ has an $n$-origin extension 
   $
   0 \to X \to M \oplus P \to Y \to 0
   $
   such that the $\Lambda$-dual sequence 
   $
   0 \to Y^\ast \to (M \oplus P)^\ast \to X^\ast \to 0
   $
   is also exact.
   \item
   The module $M$ is an object of $\mathscr{A}_n(\Lambda)$, and there exists an $n$-AB approximation 
   $
   0 \to Z \to X \xrightarrow{f} M \to 0
   $
   such that the homomorphism $f : X \to M$ is represented by monomorphisms.
\end{enumerate}
Then the implications  
$
(1) \Longleftrightarrow (2) \Longleftrightarrow (3) \Longrightarrow (4) \Longleftrightarrow (5) \Longleftrightarrow (6)
$
hold.
When the equivalent conditions $(4)$–$(6)$ hold, for any $n$-AB approximation $0 \to Z \to X \xrightarrow{f} M \to 0$ of $M$ with $f$ represented by monomorphisms, and for any left $\proj\Lambda$-approximation $s : X \to P$ of $X$, the homomorphism $\binom{f}{s}:X\to M\oplus P$ is injective and the short exact sequence
\[
0 \to X \xrightarrow{\binom{f}{s}} M \oplus P \to Y \to 0,
\]
where $Y = \Cok\binom{f}{s}$, is an $n$-origin extension of $M$.
Moreover, if $\Lambda$ satisfies the $(\ast_n)$-condition, then all the conditions are equivalent.  
In particular, when $\Lambda$ is commutative, all conditions are equivalent.
\end{thm}

\begin{proof}
The equivalence $(1) \Longleftrightarrow (2)$ is none other than \cite[Corollary~2.32]{AB}.  
The implications $(1) \Longrightarrow (3)$ and $(1) \Longrightarrow (5)$ follow from \cite[Theorem~2.37]{AB} and \cite[Theorem~2.41]{AB}, respectively.
To prove the implication $(3) \Longrightarrow (1)$, suppose that $M$ admits an $n$-spherical filtration
\[
M_n \subset M_{n-1} \subset \cdots \subset M_1 \subset M_0 = M \oplus P.
\]
For each $1 \le k \le n$, set $C_k = M_{k-1} / M_k$.
Since the $\Lambda$-dual map $M_{k-1}^\ast \to M_k^\ast$ is surjective, we obtain a long exact sequence
\[
0 \to \Ext^1(C_k, \Lambda) \to \Ext^1(M_{k-1}, \Lambda) \to \Ext^1(M_k, \Lambda) \to \Ext^2(C_k, \Lambda)\to\cdots.
\]
By assumption, we have isomorphisms $\Ext^k(C_k, \Lambda) \cong \Ext^k(M_{k-1}, \Lambda)$ and $\Ext^j(M_{k-1}, \Lambda) \cong \Ext^j(M_k, \Lambda)$ for all $j > k$.
Therefore, for every $1 \le i \le n$, we obtain
\[
\Ext^i(M, \Lambda) \cong \Ext^i(M_0, \Lambda) \cong \Ext^i(M_1, \Lambda) \cong \cdots \cong \Ext^i(M_{i-1}, \Lambda) \cong \Ext^i(C_i, \Lambda).
\]
Since $C_i$ is $i$-spherical, it follows from Lemma~\ref{astn} that $\grade \Ext^i(M, \Lambda) = \grade \Ext^i(C_i, \Lambda) \ge i$.

Let us prove the equivalence between the conditions $(4)$, $(5)$ and $(6)$.
The implication $(5)\Longrightarrow(4)$ clearly holds.
Next, we prove the implication $(4)\Longrightarrow(6)$.
Assume the condition (4) holds.
Then $M$ has an $n$-origin extension  
$
0 \to X \xrightarrow{\binom{f}{s}} M \oplus P \to Y \to 0.
$  
Taking a syzygy of $Y$, we obtain a short exact sequence  
$
0 \to \syz Y \to X \oplus Q \xrightarrow{(f, t)} M \to 0,
$  
with $Q$ projective.
Since $f$ is represented by monomorphisms, this sequence gives the desired $n$-AB approximation.
Conversely, assume condition $(6)$.  
Then there exists an $n$-AB approximation $0 \to Z \to X \xrightarrow{f} M \to 0$ such that $f$ is represented by monomorphisms.  
Since $s : X \to P$ is a left $\proj \Lambda$-approximation of $X$, it follows from \cite[Theorem~3.3 and Lemma~3.4]{Omorp} that the homomorphism  
$
\binom{f}{s} : X \to M \oplus P
$
is an injection.  
We obtain the following commutative diagram with exact rows and columns:
$$
\xymatrix@R-1pc@C-1pc{
&&0\ar[d]&&0\ar[d]&&&&\\
0\ar[rr]&&Z\ar[rr]\ar[dd]&&P\ar[rr]\ar[dd]^-{\binom{0}{1}}&&Y\ar[rr]\ar@{=}[dd]&&0\\
&&&&&&&\\
0\ar[rr]&&X\ar[rr]^-{\binom{f}{s}}\ar[dd]^-{f}&&M\oplus P\ar[dd]^-{(1, 0)}\ar[rr]&&Y\ar[rr]&&0\\
&&&&&&&\\
&&M\ar@{=}[rr]\ar[d]&&M\ar[d]&&\\
&&0&&0,&&
}
$$
where $Y = \Cok\binom{f}{s}$.
From the first row, we see that $Y$ has projective dimension at most $n$, and the second row
$
0 \to X \xrightarrow{\binom{f}{s}} M \oplus P \to Y \to 0
$
is an $n$-origin extension of $M$\footnote{The module $\Cok\binom{f}{s}$ is called the \emph{stable cokernel} of the morphism $f$, and is denoted by $\Cokv f$; see \cite{Kkern, Kmorp, Omorp}.}.  
Also, since $s : X \to P$ is a left $\proj\Lambda$-approximation of $X$, the $\Lambda$-dual homomorphism
$
{\binom{f}{s}}^\ast : (M \oplus P)^\ast \to X^\ast
$
is surjective.  
Thus, the condition (5) holds.

Finally, let us assume that $\Lambda$ satisfies the condition $(\ast_n)$, and let $M$ be an object of $\mathscr{E}_n(\Lambda)$.  
By the condition (5), there exists an $n$-origin extension  
$
0 \to X \to M \oplus P \to Y \to 0
$
of $M$ such that the $\Lambda$-dual homomorphism $(M \oplus P)^\ast \to X^\ast$ is surjective.  
Then, from the long exact sequence of $\Ext$, we obtain isomorphisms
$
\Ext^i(M, \Lambda) \cong \Ext^i(Y, \Lambda)
$
for all $1 \le i \le n$.  
Since $Y$ has projective dimension at most $n$, the assumption implies that
$
\grade \Ext^i(M, \Lambda) = \grade \Ext^i(Y, \Lambda) \ge i
$
for all $1\le i\le n$.
\end{proof}

As stated in the theorems above, the existence of $n$-AB approximations and $n$-origin extensions is deeply connected to the grade of Ext modules.  
From this, we obtain the following results concerning the properties of the categories $\mathscr{A}_n(\Lambda)$ and $\mathscr{E}_n(\Lambda)$.

\begin{rmk}\label{excl}
Let $n>0$ be an integer.
\begin{enumerate}
   \item
   Let $R$ be a commutative noetherian ring.
Then, the subcategories $\mathscr{A}_n(R)$ and $\mathscr{C}’_n(R)$ of $\mod R$ are closed under extensions.
Indeed, by Theorem \ref{thmnABapp}, a finitely generated $R$-module $M$ is in $\mathscr{A}_n(R)$ if and only if $\grade\Ext^i(M,R)\ge i-1$ for all $1\le i\le n$.
For a short exact sequence $0\to X\to Y\to Z\to0$ in $\mod R$ with $X$, $Z\in\mathscr{A}_n(R)$, we have an exact sequence $\Ext^i(Z,R)\to \Ext^i(Y,R)\to\Ext^i(X,R)$ for any $i$.
When the Ext modules on both ends vanish by localization at a prime ideal $\p$ of $R$, the same holds for the middle term.  
Hence, we have $\grade\Ext^i(Y, R)\ge i-1$ for all $1\le i\le n$, and $Y\in\mathscr{A}_n(R)$.  
The same argument applies to $\mathscr{E}_n(R)$.
On the other hand, $\mathscr{H}_n(R)$ is, in general, not closed under extensions.  
For example, if $R$ is a non-Gorenstein artinian ring, then $\mathscr{H}_1(R)$, the subcategory of torsionless modules in $\mod R$, is not closed under extensions.
   \item 
   When do the subcategories $\mathscr{A}_n(\Lambda)$ and $\mathscr{E}_n(\Lambda)$ coincide with $\mod\Lambda$?  
   The stronger conditions considered in Theorems~\ref{thmnABapp} and~\ref{norgeq} above have been studied.  
   Namely, we consider the following:
   \begin{itemize}
      \item[(i)]
      One has the inequality $\grade\Ext^i(M,\Lambda)\ge i-1$ for all $M\in\mod\Lambda$ and $1\le i\le n$.
      \item[(ii)]
      One has the inequality $\grade\Ext^i(M,\Lambda)\ge i$ for all $M\in\mod\Lambda$ and $1\le i\le n$.
   \end{itemize}
   Clearly, the condition (ii) implies (i).  
   The condition (i) (resp.\ (ii)) is denoted by $(\mathrm{d}_n)$ (resp.\ $(\mathrm{b}_n)$) in \cite{AR1}, and by $g_n(0)$ (resp.\ $g_n(1)$) in \cite{HI}.  
   As summarized in \cite{HI}, an $n$-Gorenstein ring (in the sense of Auslander) satisfies the condition (ii).  
   Moreover, it is stated there that condition (ii) for $\Lambda$ is equivalent to $\Lambda^{\mathrm{op}}$ being quasi $n$-Gorenstein.
   See also \cite{IyaSym}.
   For a commutative noetherian ring $R$, by \cite[Proposition 2.41]{AB}, the conditions (i) and (ii) are equivalent to the following one.
   \begin{itemize}
   \item[(iii)]
   For any prime ideal $\p$ of $R$ with $\depth R_{\p}\le n-1$, the localization $R_\p$ is Gorenstein.
   \end{itemize}
\end{enumerate}
\end{rmk}

The situation where Ext modules have higher grade is considered in \cite[Section 4]{Omorp}.
As a consequence, the following proposition is obtained.

\begin{prop}\label{highgrade}
Let $M$ be a finitely generated $\Lambda$-module, and let $n, k\ge0$ be integers.
Consider the following conditions.
\begin{enumerate}[\rm(1)]
   \item
   One has the inequality $\grade_{\Lambda^{\mathrm{op}}}\Ext^i_\Lambda(M,\Lambda)\ge i+k$ for all $1\le i\le n$.
   \item
   The module $M$ has an $n$-origin extension $0\to X\to M\oplus P\to Y\to0$ such that $Y$ is $k$-torsionfree.
\item
   The module $M$ has an $n$-origin extension $0\to X\to M\oplus P\to Y\to0$ such that $Y$ is $k$-torsionfree and the $\Lambda$-dual sequence $0\to Y^\ast\to (M\oplus P)^\ast\to X^\ast\to 0$ is also exact.
\end{enumerate}
Then the implications $(1)\Longrightarrow(2)\Longleftrightarrow(3)$ hold.
Moreover, if $\Lambda$ satisfies the $(\ast_{n+k})$-condition, then all conditions are equivalent.
In particular, when $\Lambda$ is commutative, all conditions are equivalent.
\end{prop}

\begin{proof}
The implications $(1)\Longrightarrow(2)\Longleftrightarrow(3)$ follow from \cite[Theorems 3.8 and 4.6]{Omorp}.
Assume that $\Lambda$ satisfies the $(\ast_{n+k})$-condition and that condition (3) holds. Then, since $Y$ is $k$-torsionfree, it follows from \cite[Theorem 2.17]{AB} that there exists $Y' \in \mod \Lambda$ such that $Y$ is the $k$-th syzygy of $Y'$.
Since $Y$ has projective dimension at most $n$, $Y'$ has projective dimension at most $n+k$.
In this case, by considering the long exact sequence of Ext, we obtain isomorphisms $\Ext^i(M,\Lambda)\cong\Ext^i(Y,\Lambda)\cong\Ext^{i+k}(Y’,\Lambda)$ for all $1 \le i \le n$.
Therefore, we can conclude that $\grade\Ext^i(M,\Lambda)=\grade\Ext^{i+k}(Y’,\Lambda)\ge i+k$ for all $1 \le i \le n$.

\end{proof}

Finally, we discuss the relationship between the existence of $n$-FPD hulls and higher torsionfreeness.
The following theorem is essentially obtained in \cite[Theorem 2.3]{OExt} and its proof.
However, since that paper assumes commutativity of the ring, we give the precise statement in our general setting.

\begin{thm}\label{nFPDhullex}
Let $M$ be a finitely generated $\Lambda$-module and $n \ge 0$ an integer.
Then the following are equivalent.
\begin{enumerate}[\rm(1)]
   \item 
   The $n$-th syzygy $\syz^n M$ is $(n+1)$-torsionfree.
   \item 
   The module $M$ is an object of $\mathscr{H}_n(\Lambda)$, i.e., $M$ has an $n$-FPD hull.
   \item 
   The module $M$ is an object of $\mathscr{A}_n(\Lambda)$, and there exists an $n$-AB approximation $0\to Z\to W\to M\to 0$ of $M$ such that $W$ is torsionless.
   \item 
   The module $M$ is an object of $\mathscr{E}_n(\Lambda)$, and there exists an $n$-origin extension $0\to X'\to M\oplus P\to Y'\to 0$ of $M$ such that $X'$ is torsionless.
\end{enumerate}
When the equivalent conditions above hold, for any $n$-AB approximation $0\to Z\to W\xrightarrow{f} M\to0$ with $W$ torsionless and any left $\proj\Lambda$-approximation $s:W\to P$ of $W$, the exact sequence
$$
0\to M\to Y\to X\to 0
$$
obtained from the pushout diagram, where $Y=\Cok\binom{f}{s}$ and $X=\Cok s$, is an $n$-FPD hull of $M$.

\end{thm}

\begin{proof}
The implications $(2) \Longrightarrow (4) \Longrightarrow (3)$ follow by taking syzygies of the modules on right-hand side of the sequence.
The implication $(3) \Longrightarrow (1)$ and the construction of the $n$-FPD hull in this case follow from the proof of the implication $(1) \Longrightarrow (2)$ in \cite[Theorem 2.3]{OExt}.
Note that the proof of the equivalence of the conditions (1), (2), and (3) in \cite[Theorem 2.3]{OExt} does not rely on the commutativity of the ring.
The equivalence $(1) \Longleftrightarrow (2)$ in our theorem is precisely the same as the equivalence of conditions (1) and (2) in \cite[Theorem 2.3]{OExt}.
\end{proof}

As a consequence of the theorems above, we obtain the following inclusions among our subcategories. For a subcategory $\mathscr{X}$ of $\mod \Lambda$, we denote by $\syz^n \mathscr{X}$ the subcategory consisting of the $n$-th syzygies of objects in $\mathscr{X}$ for $n\ge0$.

\begin{cor}\label{incl}
Let $n\ge0$ be an integer.
\begin{enumerate}[\rm(1)]
   \item
   We have the following descending chain of subcategories of $\mod\Lambda$.
   $$
   \mathscr{A}_n(\Lambda)\supset \mathscr{C}’_n(\Lambda)\supset \mathscr{H}_n(\Lambda)\supset \syz\mathscr{A}_{n+1}(\Lambda)\supset\syz\mathscr{C}’_{n+1}(\Lambda)\supset \syz\mathscr{H}_{n+1}(\Lambda)\supset\syz^2\mathscr{A}_{n+2}(\Lambda)\supset\cdots.
   $$
   \item
   Suppose that $\Lambda$ satisfies the condition $(\ast_{n-1})$ (e.g., $\Lambda$ is commutative).
   Then we have the following descending chain of subcategories of $\mod\Lambda$.
   $$
\mathscr{A}_0(\Lambda)=\mod\Lambda=\mathscr{A}_1(\Lambda)\supset\mathscr{A}_2(\Lambda)\supset\cdots\supset \mathscr{A}_{n-1}(\Lambda)\supset\mathscr{A}_n(\Lambda).
   $$
   \item
   Suppose that $\Lambda$ satisfies the condition $(\ast_n)$ (e.g. $\Lambda$ is commutative).
   Then we have the following descending chain of subcategories of $\mod\Lambda$.
   $$
   \mathscr{C}’_0(\Lambda)=\mod\Lambda\supset\mathscr{C}’_1(\Lambda)\supset\mathscr{C}’_2(\Lambda)\supset\cdots\supset \mathscr{C}’_{n-1}(\Lambda)\supset\mathscr{C}’_n(\Lambda).
   $$
\end{enumerate}
\end{cor}

\begin{proof}
It is enough to prove the statement (1) that the inclusion $\mathscr{H}_n(\Lambda)\supset\syz\mathscr{A}_{n+1}(\Lambda)$ holds true.
But, this is a direct consequence of the equivalences $(3)\Longleftrightarrow(4)$ in Theorem \ref{thmnABapp} and $(1)\Longleftrightarrow(2)$ in Theorem \ref{nFPDhullex}.
The statement (2) follows from the assumption and Theorem~\ref{thmnABapp}.
In fact, in this situation, for any $1 \le k \le n$ and any $M \in \mod \Lambda$, the conditions $M \in \mathscr{A}_k(\Lambda)$ and $\grade\Ext^i(M, \Lambda) \ge i - 1$ for all $1 \le i \le k$ are equivalent.
The statement (3) can be proved similarly.
\end{proof}

At the end of this section, we state some remarks on the category $\mathscr{H}_n(\Lambda)$.
The existence of $n$-FPD hull is closely related to the delooping level of modules introduced by G\'{e}linas \cite{Gel}.

\begin{rmk}\label{dell}
\begin{enumerate}
    \item
    By the above corollary, if $\Lambda$ satisfies the condition $(\ast_n)$, then the inclusions
$\mathscr{A}_n(\Lambda)\supset\mathscr{A}_{n-1}(\Lambda)$ and $\mathscr{E}_n(\Lambda)\supset\mathscr{E}_{n-1}(\Lambda)$ hold.
   The same is not necessarily true for $\mathscr{H}_n(\Lambda)$, even when $\Lambda$ is commutative.
   Indeed, let $R$ be a commutative noetherian local ring with residue field $k$, and assume that $t = \depth R > 0$.
   Then, by \cite[Theorem 3.1]{Gel}, the module $\syz^t k$ is $(t+1)$-torsionfree.
   However, it is easily seen that $\syz^i k$ is not $(i+1)$-torsionfree for any $0 \le i < t$.
   \item 
   G\'{e}linas \cite{Gel} defined the {\em delooping level} of a finitely generated $\Lambda$-module $M$, which is denoted by $\dell_\Lambda M$, as follows.
   $$
\dell_\Lambda M=\inf\{n\ge0 \mid \syz^n M\text{ is a direct summand of }\syz^{n+1}N\text{ in }\lmod\Lambda\text{ for some }N\in\mod\Lambda\}.
$$
He also defined the delooping level $\dell\Lambda$ of the ring $\Lambda$ by $\dell\Lambda=\sup\{\dell_\Lambda S\mid S\text{ is simple}\}$.
He proved that if $\Lambda$ satisfies a certain grade condition, then the (little) finitistic dimension of $\Lambda^{\mathrm{op}}$ coincides with $\dell\Lambda$.
Thus, the delooping level carries important information about the ring and has been actively studied.
The delooping level is related to the existence of $n$-FPD hulls in the following way:
$$
\text{If }M\text{ has }n\text{-FPD hull, then the inequality }\dell_{\Lambda}M\le n\text{ holds}.
$$
This follows from Theorem \ref{nFPDhullex} and \cite[Theorem 1.13]{Gel}.
\end{enumerate}
\end{rmk}

\section{Approximating sequences for Martsinkovsky's $\xi$-invariants}
The $\delta$-invariant for finitely generated modules over a commutative Gorenstein local ring was introduced by Auslander in his unpublished paper \cite{AusM}.
Martsinkovsky \cite{Mart96I} later defined the $\xi$-invariant over general commutative noetherian local rings, so that it coincides with Auslander's $\delta$-invariant when the ring is Gorenstein.
In this section, we study a certain subsequence of the $\xi$-invariant.
To state the definition, we first need some preparation.

Let $\Lambda$ be a two-sided noetherian ring, and let $M$, $N\in\mod\Lambda$.
In this setting, we consider the following sequence of maps induced by taking syzygies:
$$
\lhom_\Lambda(M,N)\xlongrightarrow{\syz}\lhom_\Lambda(\syz M,\syz N)\xlongrightarrow{\syz}\cdots\xlongrightarrow{\syz}\lhom_\Lambda(\syz^n M,\syz^n N)\xlongrightarrow{\syz}\lhom_\Lambda(\syz^{n+1}M,\syz^{n+1}N)\xlongrightarrow{\syz}\cdots
$$
Let $\pi$ be the natural surjection $\Hom_\Lambda(M,N) \twoheadrightarrow \lhom_\Lambda(M,N)$.  
For an integer $n \ge 0$, we define a subgroup $\V_n(M,N)$ of $\Hom_\Lambda(M,N)$ as the kernel of the composition
$$
\Hom_\Lambda(M,N) \xrightarrow{\pi} \lhom_\Lambda(M,N) \xrightarrow{\syz^n} \lhom_\Lambda(\syz^n M, \syz^n N).
$$
That is,
$$
\V_n(M,N) = \{ f \in \Hom_\Lambda(M,N) \mid \underline{\syz^n f} = \underline{0}\text{ in } \lhom_\Lambda(\syz^nM,\syz^nN)\}.
$$
We state here some properties of the subgroup $\V_n(M,N)$, which will be used later.

\begin{lem}\label{Vlem}
Let $M$, $M^\prime$, $N$ and $N^\prime$ be finitely generated $\Lambda$-modules, and let $n\ge0$ be an integer.
Then the following hold.
\begin{enumerate}[\rm(1)]
   \item 
   One has $\V_n(M\oplus M^\prime,N)=\V_n(M,N)\oplus\V_n(M^\prime,N)$ and $\V_n(M,N\oplus N^\prime)=\V_n(M,N)\oplus\V_n(M,N^\prime)$.
   \item 
   Assume that either $M$ or $N$ has projective dimension at most $n$.
   Then one has $\V_n(M,N)=\Hom_\Lambda(M,N)$.
   \item 
   Suppose that there is a surjection $M\twoheadrightarrow N$ of $\Lambda$-modules.
   Then, for any finitely generated $\Lambda$-module $T$, the abelian group $\V_n(N,T)$ can be regarded as a subgroup of the abelian group $\V_n(M,T)$.
\end{enumerate}
\end{lem}

\begin{proof}
\begin{enumerate}[\rm(1)]
   \item 
   The claim follows from the additivity of the natural composition map $$
   \Hom(M,N) \to \lhom(M,N) \to \lhom(\syz^n M, \syz^n N) $$ with respect to $M$ and $N$.
   \item
   Since we have $\lhom(\syz^n M, \syz^n N) = 0$ by assumption, the claim follows immediately from the definition.
   \item
   The desired conclusion follows from the following commutative diagram:
\end{enumerate}
   $$
   \xymatrix@R-1pc@C-1pc{
   \Hom(N,T) \ar@{^{(}->}[rr]\ar[dd]^{\text{nat.}} && \Hom(M,T) \ar[dd]^{\text{nat}} \\
   &&&\\
   \lhom(\syz^nN,\syz^nT) \ar[rr] && \lhom(\syz^nM,\syz^nT).
   }
   $$
\end{proof}

We define a subgroup $\V(M,N)$ of $\Hom_\Lambda(M,N)$ as the union $\bigcup_{n \ge 0} \V_n(M,N)$, and obtain an ascending chain of subgroups:
\begin{equation}\label{eq:ascend}
\V_0(M,N) \subset \V_1(M,N) \subset \cdots \subset \V_n(M,N) \subset \V_{n+1}(M,N) \subset \cdots \subset \V(M,N).
\end{equation}
Note that the ascending chain \eqref{eq:ascend} stabilizes under suitable assumptions.  
For example, if $R$ is a commutative noetherian ring and $\Lambda$ is a noetherian $R$-algebra, then the acceding chain \eqref{eq:ascend} is a sequence of $R$-submodules of the finitely generated $R$-module $\Hom_\Lambda(M,N)$.
Hence, there exists an integer $m=m(M,N) \ge 0$ such that $\V_\ell(M,N) = \V_m(M,N)$ for all integers $\ell \ge m$.  
In particular, we have $\V(M,N) = \V_m(M,N)$.

Alternatively, when a $\Lambda$-module $M$ is Ext-orthogonal to the ring $\Lambda$, the following result on the stabilization of $\V_*(M, -)$ holds.

\begin{lem}\label{VExtortho}
Let $a$ and $b$ be integers with $0 < a \le b$.  
Suppose that $M \in \mod \Lambda$ satisfies $\Ext_\Lambda^i(M,\Lambda) = 0$ for all $i$ with $a\le i\le b$.
Then, for any $N \in \mod \Lambda$, the following equalities hold:
$$
\V_{a-1}(M, N) = \V_a(M, N) = \cdots = \V_{b-1}(M, N) = \V_b(M, N).
$$
\end{lem}

\begin{proof}
The syzygy module $\syz^{a-1}M$ satisfies $\Ext^i_\Lambda(\syz^{a-1}M, \Lambda) = 0$ for all $1 \le i \le b - a + 1$.  
Hence, by the dual statement of \cite[Theorem 2.17]{AB}, the counit morphism $\underline{\psi^{b-a+1}_{\syz^{a-1}M}} : \tr\syz^{b-a+1}\tr\syz^b M \to \syz^{a-1}M$ is an isomorphism in $\lmod\Lambda$.  
Therefore, the claim follows from the following commutative diagram.
$$
\xymatrix@R-1pc@C-1pc{
\lhom(\syz^{a-1}M,\syz^{a-1}N) \ar[rrrrrrr]^-{\syz^{b-a+1}} \ar[dd]_-{\lhom(\underline{\psi^{b-a+1}_{\syz^{a-1}M}},\syz^nN)}^-{\cong} &&&&&&& 
\lhom(\syz^bM,\syz^bN) \ar@{=}[dd] \\
\\
\lhom(\tr\syz^{b-a+1}\tr\syz^bM,\syz^{a-1}N) \ar[rrrrrrr]_-{\theta^{b-a+1}_{\syz^bM,\syz^{a-1}N}}^-{\cong} &&&&&&& 
\lhom(\syz^bM,\syz^bN)
}
$$
\end{proof}

In the remainder of this section, we further assume that $R$ is a commutative noetherian local ring with unique maximal ideal $\m$ and residue field $k$, and consider only the case $\Lambda = R$.  
Under the above notation, we denote $\V_n(M,k)$ by $\V_n(M)$ and $\V(M,k)$ by $\V(M)$.  
We are now ready to define the $\xi$-invariant of a finitely generated $R$-module $M$.

\begin{dfn}\cite{Mart96I, Mart96rem}\label{xidef}
For a finitely generated $R$-module $M$, we define $\xi_R(M) = \dim_k \V(M)$ and call it the {\em Martsinkovsky's $\xi$-invariant} of $M$.  
That is, $\xi_R(M)$ is nothing but the dimension of the $k$-vector subspace
$$
\V(M) = \{ f \in \Hom_R(M,k) \mid \underline{\syz^m f} = \underline{0} \text{ in } \lhom_R(\syz^m M, \syz^m k) \text{ for some } m \ge 0 \}
$$
of $\Hom_R(M,k)$.
\end{dfn}

\begin{rmk}
The above definition may look slightly different from Martsinkovsky's original one.  
In \cite{Mart96I}, for each integer $i \ge 0$, $\xi_R^i(M)$ is defined as the dimension of the kernel of the natural homomorphism $\Ext^i_R(M,k) \to \widecheck{\Ext}^i_R(M,k)$ as a $k$-vector space, and $\xi_R^0(M)$ is denoted by $\xi_R(M)$.  
Here, $\widecheck{\Ext}^i_R(M,k)$ denotes the {\em $i$-th Tate--Vogel cohomology module}.
When $\syz^i M$ is defined to be the $i$th syzygy of $M$ arising from the minimal free resolution, it is shown that $\xi_R^i(M) = \xi_R(\syz^i M)$.
In this paper, to avoid confusion between $\xi_R^i(M)$ and $\xi_R(i, M)$ (to be defined below), we do not use the former notation.
The invariant $\xi_R(M) = \xi_R^0(M)$ defined in \cite{Mart96I} agrees with the one given in Definition \ref{xidef}.
In \cite{Mart96rem}, the subspace $\V(M)$ of $\Hom_R(M,k)$ is defined as follows:  
Let $P^\bullet_M$ and $P^\bullet_k$ be minimal free resolutions of $M$ and $k$, respectively.
Then $\V(M)$ is defined as the set of all $R$-homomorphisms $f : M \to k$ satisfying that there exists a complex homomorphism $\tilde{f}^\bullet : P^\bullet_M \to P^\bullet_k$ with $\tilde{f}^m = 0$ for all $m \gg 0$ and $H^0(\tilde{f}^\bullet) = f$.
It is easy to see that the subspace $\V(M)$ given in this way also agrees with the one described above.
\end{rmk}

Computing the $\xi$-invariant in practice is quite difficult, since it requires considering arbitrarily high syzygies of morphisms.
Thus, we consider an approximating sequence for the $\xi$-invariant.

\begin{dfn}
Let $M$ be a finitely generated $R$-module and $n\ge0$ an integer.
We define $\xi_R(n, M)= \dim_k \V_n(M)$ and call it the {\em $n$-th approximated $\xi$-invariant} of $M$.  
That is, $\xi_R(n, M)$ is the dimension of the $k$-vector subspace of $\Hom_R(M,k)$ consisting of homomorphisms $f \in \Hom_R(M,k)$ satisfying $\underline{\syz^n f} = \underline{0}$.
\end{dfn}

In the introduction, $\xi_R(n,M)$ is written as $\xi_n(M)$.
As is immediate from the definition, we obtain a non-decreasing sequence of integers
$$
\xi_R(0, M) \le \xi_R(1, M) \le \ldots \le \xi_R(n, M) \le \xi_R(n+1, M) \le \ldots \le \xi_R(M)
$$
and the equality $\lim_{n \to \infty} \xi_R(n, M) = \xi_R(M)$.  
Moreover, the inequality $\xi_R(M) \le \mu(M)$ clearly holds, where $\mu(M)$ is the minimal number of generators of $M$.  
Note that the equalities $\mu_R(M) = \dim_k(M \otimes_R k) = \dim_k \Hom_R(M, k)$ hold.

The following proposition is an $n$-th approximated version of certain properties that play an important role in computing the $\xi$-invariant (or Auslander's $\delta$-invariant).
See also \cite[Proposition 2.4]{Mart96I} and \cite[Lemma 5]{Mart96rem}.

\begin{prop}\label{immlem}
Let $M$, $N$ be finitely generated $R$-modules and $n\ge0$ an integer.
Then the following hold.
\begin{enumerate}[\rm(1)]
   \setcounter{enumi}{-1}
   \item
   The $0$-th approximated $\xi$-invariant $\xi_R(0,M)$ of $M$ is the rank of the largest free summand of $M$.
   \item
   One has $\xi_R(n,M\oplus N)=\xi_R(n,M)+\xi_R(n,N)$.
   \item 
   If $M$ has projective dimension at most $n$, then $\xi_R(n,M)=\mu_R(M)$.
   \item 
   If there exists a surjection $M\twoheadrightarrow N$ of $R$-modules, then $\xi_R(n,M)\ge\xi_R(n,N)$.
\end{enumerate}
\end{prop}
\begin{proof}
The statements (1) through (3) follow from parts (1) through (3) of Lemma~\ref{Vlem}, respectively.
Thus, it suffices to prove the statement (0).  
Suppose that $M = X \oplus R^{\oplus a}$, where $X$ is a stable module and $a \ge 0$.
Then, by assertions (1) and (2), we obtain the equalities $\xi(0, M) = \xi(0, X) + \xi(0, R^{\oplus a}) = \xi(0, X) + a$.
Hence, it suffices to show that $\xi(0, X) = 0$.
Assume that a nonzero $R$-homomorphism $f : X \to k$ factors through a free module; that is, there exist $b > 0$ and maps $h : X \to R^{\oplus b}$ and $g : R^{\oplus b} \to k$ such that $f = g \circ h$.
Let $g = (z_1, z_2, \ldots, z_b) : R^{\oplus b} \to k$ with $z_i \in R$, and write $h = (h_1, h_2, \ldots, h_b)^{\mathsf{T}} : X \to R^{\oplus b}$.  
Since $f$ is nonzero, there exists an element $x \in X$ such that
$$
\overline{z_1} h_1(x) + \overline{z_2} h_2(x) + \cdots + \overline{z_b} h_b(x) = \overline{1} \in k.
$$
It follows that there exists an integer $1 \le i \le b$ such that $\overline{z_i} h_i(x) \ne 0$ in $k$, and hence both $z_i$ and $h_i(x)$ are units in $R$.  
In particular, $h_i : X \to R$ is a split surjection, which contradicts the assumption that $X$ is stable.  
Therefore,
$$
\xi(0, X) = \dim_k\Ker\left(\Hom(X,k) \twoheadrightarrow \lhom(X,k)\right) = 0.
$$
\end{proof}

To describe the relationship among the $n$-AB approximations, the $n$-FPD hulls (studied in the previous section), and the $n$th approximated $\xi$-invariant, we recall the notion of minimality for such approximations and hulls.
This will also be needed to present the original definition of Auslander's $\delta$-invariant.

\begin{dfn}\label{mindef}
Let $M$ be a finitely generated $R$-module and $n\ge0$ an integer.
   \begin{enumerate}[\rm(1)]
   \item
   Suppose that $M$ admits an $n$-AB approximation 
   $0 \to Y \xrightarrow{i} X \xrightarrow{p} M \to 0$. 
   This approximation is said to be {\em minimal} when $X$ and $Y$ have no nonzero direct summand in common through $i$.  
   More precisely, this means that for any direct sum decomposition $X = X_0 \oplus X_1$ with $X_0 \subset \image i$, we have $X_0 = 0$.
   \item 
   Suppose that $M$ admits an $n$-FPD hull
   $0 \to M \xrightarrow{j} Y^\prime \xrightarrow{q} X^\prime \to 0$.
   This hull is said to be {\em minimal} when $X^\prime$ and $Y^\prime$ have no nonzero direct summand in common through $q$.  
   More precisely, this means that for any direct sum decomposition $X = X_0 \oplus X_1$ with $q^{-1}(X_0)\cap\image j=(0)$, we have $X_0 = 0$.
   \end{enumerate}
\end{dfn}

Recall that, for a Gorenstein local ring of Krull dimension $d$, a $d$-AB approximation is called a {\em maximal Cohen--Macaulay approximation} (or simply, an {\em MCM approximation}).  
Similarly, a $d$-FPD hull (resp. a $d$-origin extension) is simply called an {\em FPD hull} (resp, an {\em origin extension}).

The minimality of MCM approximations is discussed in various papers; see \cite{AusM, HS, LW, Sim, Yoshida, Yos93} for instance.
It is easy to see that right minimality (resp. left minimality) implies minimality.
Over a complete Gorenstein local ring, a proof of the converse is given in \cite[Lemma 2.2]{Yos93}.
A proof in the non-complete case is given in \cite[Lemma 2.3]{HS}.
We also note that an interesting theory on the minimality of origin extensions is developed in \cite{Kview}.
Using arguments similar to those in \cite{HS, Yos93}, one can prove analogous results for $n$-AB approximations and $n$-FPD hulls over local rings which are not necessarily Gorenstein.  
We also note that a proof avoiding the assumption of completeness was given by Simon in \cite{Sim}.

\begin{thm}\cite[Theorems 3.1 and 3.3]{Sim}\label{minimality}
Let $R$ be a local ring, $M$ a finitely generated $R$-module and $n\ge0$ an integer.
\begin{enumerate}[\rm(1)]
   \item 
   Suppose that $M$ admits an $n$-AB approximation $0 \to Y \xrightarrow{i} X \xrightarrow{p} M \to 0$.
   Then this approximation is minimal if and only if the homomorphism $p$ is right minimal.
   \item 
   Suppose that $M$ admits an $n$-FPD hull $0 \to M \xrightarrow{j} Y^\prime \xrightarrow{q} X^\prime \to 0$.
   Then this hull is minimal if and only if the homomorphism $j$ is left minimal.
\end{enumerate}
\end{thm}

We obtain an existence and uniqueness theorem for minimal $n$-AB approximations and $n$-FPD hulls.

\begin{cor}\label{minapphull}
Let $R$ be a local ring and $n\ge0$ an integer.
\begin{enumerate}[\rm(1)]
   \item
   Let $M$ be a finitely generated $R$-module belonging to $\mathscr{A}_n(R)$.  
   Then $M$ admits a minimal $n$-AB approximation
   $$
   0 \to Y_M \xrightarrow{i_M} X_M \xrightarrow{p_M} M \to 0,
   $$
   which is unique up to isomorphism of exact sequences inducing the identity on $M$.
   \item
   Let $M$ be a finitely generated $R$-module belonging to $\mathscr{H}_n(R)$.  
   Then $M$ admits a minimal $n$-FPD hull
   $$
   0 \to M \xrightarrow{j^M} Y^M \xrightarrow{q^M} X^M \to 0,
   $$
   which is unique up to isomorphism of exact sequences inducing the identity on $M$.
\end{enumerate}
\end{cor}

\begin{proof}
It suffices to prove (1), since (2) is its dual statement.  
By Theorem~\ref{thmnABapp}, $M$ admits an $n$-AB approximation $0\to Y \xrightarrow{i} X\to  M \to 0$.
By removing any direct summands common to $Y$ and $X$ via $i$, we obtain a minimal $n$-AB approximation of $M$.  
By Theorem~\ref{minimality}, minimality is equivalent to right minimality, so the lifting property and a standard argument yield the uniqueness of the minimal one up to isomorphism.
See also the proofs of \cite[Proposition~11.13]{LW} and \cite[Theorem~2.4]{Yos93}.
\end{proof}

Over a Gorenstein ring, Auslander's $\delta$-invariant is defined using the existence of minimal MCM approximations, as follows.

\begin{dfn}\cite{AusM, ADS}
Let $R$ be a Gorenstein local ring and $M$ a finitely generated $R$-module.
Take a minimal MCM approximation $0 \to Y_M \to X_M \to M \to 0$ of $M$ and consider a decomposition $X_M \cong \underline{X} \oplus F$, where $\underline{X}$ is a stable module and $F$ is a free module.  
Then the {\em $\delta$-invariant} $\delta_R(M)$ of $M$ is defined to be the rank of the free module $F$.  
In other words, $\delta_R(M)$ is the rank of the largest free summand of the minimal MCM approximation $X_M$.
\end{dfn}

The definition of the $\xi$-invariant does not depend on the existence of approximations such as MCM approximations.  
However, Martsinkovsky showed in \cite[Theorem 2.1 and Proposition 2.3]{Mart96I} that the following theorem holds over Gorenstein local rings.

\begin{thm}[Martsinkovsky]\label{Mart}
Let $R$ be a Gorenstein local ring and $M$ a finitely generated $R$-module.
Then the equality $\delta_R(M)=\xi_R(M)$ holds true.
\end{thm}

As in the theorem above, we aim to establish a method for computing the $n$-th approximated $\xi$-invariant using $n$-AB approximations and $n$-FPD hulls.
We begin with some preparations.  
When a module $M$ is Ext-orthogonal to the ring $R$, the following symmetry arises as a direct consequence of Lemma~\ref{VExtortho}.

\begin{cor}\label{symm}
Let $M$ be a finitely generated $R$-module and $n\ge0$ an integer.
\begin{enumerate}[\rm(1)]
   \item 
   Let $a$ and $b$ be integers with $0< a\le b$.
   If $\Ext^i_R(M,R)=0$ for all $a\le i\le b$, then the equalities $\xi_R(a-1,M)=\xi_R(a,M)=\cdots=\xi_R(b-1,M)=\xi_R(b,M)$ hold true.
   \item 
   Suppose $\Ext^i_R(M,R)=0$ for all $1\le i\le n$,
   Then $\xi_R(j,M)$ is the rank of the largest free summand of $M$ for any $0\le j\le n$.
\end{enumerate}
\end{cor}

\begin{proof}
The assertion (2) follows from the assertion and Proposition \ref{immlem}.
The assertion (1) is a consequence of Lemma \ref{VExtortho}.
\end{proof}

The following proposition is particularly useful in determining whether a given $n$-AB approximation is minimal.  
The original result for MCM approximations over Gorenstein local rings was given by Auslander \cite{AusM}, and an explicit proof in this case is provided in \cite[Corollary~2.10]{Yos93}.

\begin{prop}\label{nABmin}
Let $M$ be a finitely generated $R$-module and $n \ge 0$ an integer.  
Suppose that $M$ admits an $n$-AB approximation 
$0\to Y\xrightarrow{i}X\xrightarrow{p}M \to 0$.
Write $X = \underline{X} \oplus F$, where $\underline{X}$ is a stable module and $F$ is a free module, and write $p = (p_0, p_1) : X = \underline{X} \oplus F \to M$.  
Then the composite homomorphism $F \xrightarrow{p_1} M \overset{\pi}{\twoheadrightarrow} M /{\image p_0}$ is surjective, where $\pi:M\twoheadrightarrow \Cok p_0=M /{\image p_0}$ is the natural surjection.
Hence, the inequality $\mu_R(F)\ge\mu_R(\Cok p_0)$ holds.
Moreover, the following conditions are equivalent.
\begin{enumerate}[\rm(1)]
\item
The approximation $0\to Y\xrightarrow{i}X\xrightarrow{p}M \to 0$ is minimal.
\item  
The composite homomorphism $F \xrightarrow{p_1} M \overset{\pi}{\twoheadrightarrow} \Cok p_0$ is a minimal free cover of $\Cok p_0$, that is, the equality $\mu_R(F)=\mu_R(\Cok p_0)$ holds.
\end{enumerate}
\end{prop}

\begin{proof}
A similar argument to that in the proof of \cite[Lemma~2.8]{Yos93} works in this case.  
Alternatively, an explanation can also be found in \cite[Lemma~11.26]{LW}.
\end{proof}

The following theorem is an analogue of Theorem~\ref{Mart}, due to Martsinkovsky, for the $n$-th approximated $\xi$-invariant.  
Here, we remark that it is not necessary to assume the Gorenstein property of the ring, and that if $M$ is an object of $\mathscr{A}_n$, then $\xi_R(n, M)$ can be computed by using an $n$-AB approximation.

\begin{thm}\label{ABrank}
Let $M$ be a finitely generated $R$-module and $n \ge 0$ an integer.  
Assume that $M \in \mathscr{A}_n(R)$, and take a minimal $n$-AB approximation 
$0 \to Y_M \xrightarrow{i_M} X_M \xrightarrow{p_M} M \to 0$.  
Then the $n$-th approximated $\xi$-invariant $\xi_R(n, M)$ coincides with the rank of the largest free summand of $X_M$.
\end{thm}

\begin{proof}
Write $X_M = \underline{X} \oplus F$, where $\underline{X}$ is a stable module and $F$ is a free module, and write $p_M = (p_0, p_1) : \underline{X} \oplus F \to M$.  
Then $\underline{X}$ is a stable module satisfying $\Ext^i_R(\underline{X}, R) = 0$ for all $1 \le i \le n$, and $p_M : X_M \to M$ is surjective.  
Hence, by Proposition~\ref{immlem} and Corollary~\ref{symm}(2), we obtain
$$
\mu(F) = \xi(n, F) = \xi(n, \underline{X} \oplus F)=\xi(n,X_M) \ge \xi(n, M).
$$
Therefore, it suffices to show that $\xi(n, M) \ge \mu(F)$.
Let $C = \Cok p_0$, and let $\pi : M \twoheadrightarrow \Cok p_0$ denote the natural surjection.
We consider the following commutative diagram with exact rows and columns:
$$
\xymatrix@R-1pc@C-1pc{
&&0\ar[d]&&0\ar[d]&&\\
0\ar[rr]&&\V_n(C)\ar[dd]\ar[rr]&&\Hom(C,k)\ar[rr]\ar[dd]^-{\Hom(\pi,k)}&&\lhom(\syz^nC,\syz^nk)\ar[dd]^-{\lhom(\syz^n\pi,\syz^nk)}\\
&&&&&&&\\
0\ar[rr]&&\V_n(M)\ar[rr]&&\Hom(M,k)\ar[rr]&&\lhom(\syz^nM,\syz^nk).
}
$$
We will show that $\underline{\syz^n \pi} : \syz^n M \to \syz^n C$ is the zero map in the stable module category $\lmod R$.  
Indeed, in the $n$-AB approximation $0 \to Y_M \to X_M \xrightarrow{p_M} M \to 0$, the module $Y_M$ has projective dimension at most $n - 1$.  
Therefore, by considering the long exact sequence of Ext, we see that $\underline{\syz^n p_M} : \syz^n X_M \to \syz^n M$ is an isomorphism in $\lmod R$.
Also, the morphism $\underline{\binom{1}{0}} : \underline{X} \to \underline{X} \oplus F$ is an isomorphism in $\lmod R$, and since $\underline{p_0} = \underline{p_M} \circ \underline{\binom{1}{0}}$, it follows that $\underline{\syz^n p_0} : \syz^n \underline{X} \to \syz^n M$ is also an isomorphism in $\lmod R$.
Moreover, since the composition $\pi \circ p_0 : \underline{X} \to M \twoheadrightarrow C$ is the zero map, the composition  
$\underline{\syz^n \pi} \circ \underline{\syz^n p_0} = \underline{\syz^n(\pi \circ p_0)} : \syz^n \underline{X} \to \syz^n C$  
is also the zero map in $\lmod R$.  
However, since $\underline{\syz^n p_0}$ is an isomorphism, we conclude that $\underline{\syz^n \pi} = \underline{0}$ in $\lhom(\syz^n M, \syz^n C)$.

Consequently, the homomorphism $\lhom(\syz^n \pi, \syz^n k)$ is the zero map, and the homomorphism $\Hom(\pi, k)$ induces an injective homomorphism $\Hom(C, k) \hookrightarrow \V_n(M)$.  
On the other hand, by Proposition~\ref{nABmin}, we have the equality $\mu(F) = \mu(C)$.  
Therefore, we obtain
$$
\xi(n, M) = \mu(\V_n(M)) \ge \mu(C) = \mu(F),
$$
and the proof is completed.
\end{proof}

The following theorem was proved by Auslander \cite{AusM} and announced at the Berkeley symposium on commutative algebra held at MSRI in 1987.

\begin{thm}[Auslander]\label{AusBerk}
Let $R$ be a Gorenstein local ring and $M$ a finitely generated $R$-module.
Then $\delta_R(M)=0$ if and only if every morphism $f:M\to Z$ in $\mod R$ with $\pd_RZ<\infty$ satisfies $f\otimes_R k=0$.
\end{thm}

The theorem above given by Auslander is generalized as follows.  
Known proofs of the above theorem, such as those described in \cite[Lemma~2.11 and Corollary~2.12]{Yos93} or \cite[Lemma~2.1 and Corollary~2.3]{Kfree}, rely crucially on the existence of FPD hulls.  
However, the following theorem can be proved under the sole assumption that the module is an object of $\mathscr{E}_n'(R)$; therefore it is not necessarily a submodule of a module of finite projective dimension.
Note that the implication \textup{(1)} $\Longrightarrow$ \textup{(2)} in the following theorem is a refinement of \cite[Lemma~9]{Mart96rem}.

\begin{thm}\label{orginm}
Let $M$ be a finitely generated $R$-module and $n \ge 0$ an integer.  
Consider the following conditions:
\begin{enumerate}[\rm(1)]
   \item
   One has $\xi_R(n, M) = 0$.
   \item
   Every morphism $f : M \to Z$ in $\mod R$ with $\pd_R Z \le n$ satisfies $f \otimes_R k = 0$.
\end{enumerate}
Then the implication \textup{(1)} $\Longrightarrow$ \textup{(2)} always holds. 
The converse also holds if $M$ is an object of $\mathscr{E}_n'(R)$.
\end{thm}

\begin{proof}
Let us prove the implication \textup{(1)} $\Longrightarrow$ \textup{(2)}.  
Fix a morphism $f : M \to Z$ in $\mod R$ with $\pd Z \le n$.  
We consider the natural surjection $\pi : Z \twoheadrightarrow Z /{\m Z}$, and it suffices to show that the composition $M \xrightarrow{f} Z \xrightarrow{\pi} Z /{\m Z}$ is the zero map.
We write $\pi = (\pi_1, \pi_2, \ldots, \pi_r)^{\mathsf{T}} : M \to Z \cong k^{\oplus r}$ via the isomorphism $Z/\m Z \cong k^{\oplus r}$, where $r = \mu(Z)$.
For each $1 \le i \le r$, since $\syz^n Z$ is a free module, the morphism  
$\underline{\syz^n(\pi_i \circ f)} = \underline{\syz^n \pi_i} \circ \underline{\syz^n f} : \syz^n M \to \syz^n Z \to \syz^n k$  
is zero in $\lmod R$.
Moreover, the assumption $\xi(n, M) = 0$ implies that the homomorphism $\Hom(M, k) \to \lhom(\syz^n M, \syz^n k)$ is injective.  
Therefore, for each $1 \le i \le r$, the composition $\pi_i \circ f : M \to Z \to k$ is zero in $\mod R$, and we conclude that $\pi \circ f = 0$.

Next, assume that $M \in \mathscr{E}_n'(R)$, and let us prove the implication \textup{(2)} $\Longrightarrow$ \textup{(1)}.  
In particular, $M$ is an object of $\mathscr{A}_n(R)$, so by Corollary~\ref{minapphull}, there exists a minimal $n$-AB approximation  
$0 \to Y_M \xrightarrow{i_M} X_M \xrightarrow{p_M} M \to 0$.  
By Theorem~\ref{ABrank}, there is an isomorphism $X_M \cong \underline{X} \oplus R^{\oplus \xi}$, where $\underline{X}$ is a stable module and $\xi = \xi(n, M)$.
Take a left $\proj R$-approximation $s_0 : \underline{X} \to R^{\oplus m}$ of $\underline{X}$ and consider an exact sequence $\underline{X} \xrightarrow{s_0} R^{\oplus m} \to C \to 0$, where $C = \Cok s_0$.  
Since $\underline{X}$ is stable, it follows from Proposition~\ref{immlem}\textup{(0)} or its proof that the map $s_0 \otimes k : \underline{X} \otimes k \to R^{\oplus m} \otimes k$ is the zero map, that is, we have $m = \mu(C)$.  
By taking the direct sum of the free module $R^{\oplus \xi}$ with the first two terms of the exact sequence, we obtain an exact sequence  
$X_M \xrightarrow{s} R^{\oplus (m + \xi)} \to C \to 0$.  
Note that the homomorphism $s = \begin{pmatrix} s_0 & 0 \\ 0 & 1_{R^{\oplus \xi}} \end{pmatrix} : X_M = \underline{X} \oplus R^{\oplus \xi} \to R^{\oplus (m + \xi)}$  
is also a left $\proj R$-approximation of $X_M$.
Then, since the assumption $M \in \mathscr{E}_n(R)$ implies that $p_M : X_M \to M$ is represented by monomorphisms by Theorem \ref{norgeq}, and since $s : X_M \to R^{\oplus (m + \xi)}$ is a left $\proj R$-approximation of $X_M$, we obtain from \cite[Theorem~3.3 and Lemma~3.4]{Omorp} that the homomorphism  
$\binom{p_M}{s} : X_M \to M \oplus R^{\oplus (m + \xi)}$ is an injection.
We consider the following commutative diagram with exact rows and columns, which appears in the proof of Theorem \ref{norgeq} as well:
$$
\xymatrix@R-1pc@C-1pc{
&&0\ar[d]&&0\ar[d]&&&&\\
0\ar[rr]&&Y_M\ar[rr]^-{s\circ i_M}\ar[dd]^-{i_M}&&F\ar[rr]\ar[dd]^-{\binom{0}{1}}&&\Cok\binom{p_M}{s}\ar[rr]\ar@{=}[dd]&&0\\
&&&&&&&\\
0\ar[rr]&&X_M\ar[rr]^-{\binom{p_M}{s}}\ar[dd]^-{p_M}&&M\oplus F\ar[dd]^-{(1, 0)}\ar[rr]&&\Cok\binom{p_M}{s}\ar[rr]&&0\\
&&&&&&&\\
&&M\ar@{=}[rr]\ar[d]&&M\ar[d]&&\\
&&0&&0,&&
}
$$
where the free module $R^{\oplus(m+\xi)}$ is denoted by $F$.
From the first row, the module $\Cok\binom{p_M}{s}$ has projective dimension at most $n$.

We will prove that the homomorphism $F \to \Cok\binom{p_M}{s}$ is a minimal free cover.  
If not, then there exists a split monomorphism $b : R \to F$ such that $s \circ i_M = (a, b) : Y_M = Y' \oplus R \to F$.  
In this case, the upper left part of the above commutative diagram decomposes as follows.
$$
\xymatrix@R-1pc@C-1pc{
Y_M\ar@{=}[r]&Y'\oplus R\ar[rr]^-{(a,b)}\ar[dd]_-{i_M=(i_0,i_1)}&&F\ar[dd]^-{\binom{0}{1}}\\
&&&&\\
&X_M\ar[rr]^-{\binom{p_M}{s}}&&M\oplus F.
}
$$
In particular, $b = s \circ i_1$ holds.  
Since $b$ is a split injection, so is $i_1$.  
Hence, via the morphism $i_M = (i_0, i_1) : Y_M \to X_M$, the modules $X_M$ and $Y_M$ share a common direct summand $R$.  
This contradicts the minimality of our approximation.  
Therefore, the homomorphism $F \to \Cok\binom{p_M}{s}$ is a minimal free cover.
Furthermore, we obtain the following pushout diagram.
$$
\xymatrix@R-1pc@C-1pc{
X_M\ar[rr]^-{s}\ar[d]^-{p_M}&&F\ar[rr]\ar[d]&&C\ar@{=}[d]\ar[rr]&&0\\
M\ar[rr]^-{f}&&\Cok\binom{p_M}{s}\ar[rr]&&C\ar[rr]&&0.
}
$$
Since the module $\Cok\binom{p_M}{s}$ has projective dimension at most $n$, our assumption \textup{(2)} implies that $f \otimes k = 0$.  
This yields an isomorphism $\Cok\binom{p_M}{s} \xrightarrow{\cong} C \otimes k$ by the exactness of the second row of the above pushout diagram.
On the other hand, since the homomorphism $F \to \Cok\binom{p_M}{s}$ is a minimal free cover, we have an isomorphism $F \otimes k \xrightarrow{\cong} \Cok\binom{p_M}{s}$.  
Consequently, the equalities  
$$
m = \mu(C) = \mu(\Cok\left(\begin{smallmatrix} p_M \\ s \end{smallmatrix}\right)) = \mu(F) = m + \xi
$$  
hold, and we conclude that $\xi(n, M) = \xi = 0$.
\end{proof}

Over a Gorenstein local ring, it was shown by Auslander \cite{AusM} that the $\delta$-invariant can be computed using an FPD hull.  
Proofs and explanations can be found, for example, in \cite[Lemma~2.11]{Yos93} and \cite[Proposition~11.36]{LW}.  
Similarly, the $n$-th approximated $\xi$-invariant can also be computed using $n$-FPD hulls.

\begin{prop}\label{FPDxi}
Let $n \ge 0$ be an integer, and let $M$ be an object of $\mathscr{H}_n(R)$.  
Then, for any $n$-FPD hull $0 \to Y \to X \to M \to 0$, the equality $\xi_R(n, M) = \mu_R(Y) - \mu_R(X)$ holds.
\end{prop}

\begin{proof}
Essentially, the same argument as in the proof of \cite[Lemma~2.11]{Yos93} applies.  
For the reader's convenience, we give a proof.

Let $0 \to M \to Y^M \to X^M \to 0$ be the minimal $n$-FPD hull of $M$.  
First of all, we remark that for any $n$-FPD hull $0 \to M \to Y \to X \to 0$, the equality $\mu(Y^M) - \mu(X^M) = \mu(Y) - \mu(X)$ holds.
This is because, by Corollary~\ref{minapphull}, a minimal $n$-FPD hull can be obtained by removing common direct summands of $X$ and $Y$ via the homomorphism $Y \to X$.

For the minimal $n$-AB approximation $0 \to Y_M \xrightarrow{i_M} X_M \xrightarrow{p_M} M \to 0$, it follows from Theorem~\ref{ABrank} that there is a decomposition $X_M \cong \underline{X} \oplus R^{\oplus \xi}$, where $\underline{X}$ is stable and $\xi = \xi(n, M)$.  
Since $M \in \mathscr{H}_n(R)$, the module $X_M$ is torsionless, and so is $\underline{X}$.  
Therefore, there exists a short exact sequence $0 \to \underline{X} \to R^{\oplus m} \to X' \to 0$.  
As $\underline{X}$ is stable, we have $m = \mu(X')$.
We now consider the following pushout diagram:
$$
\xymatrix@R-1pc@C-1pc{
&&0\ar[d]&&0\ar[d]&&\\
&&Y_M\ar@{=}[rr]\ar[dd]^-{i_M}&&Y_M\ar[dd]&&\\
&&&&&\\
0\ar[rr]&&X_M\ar[rr]\ar[dd]^-{p_M}&&R^{\oplus (m+\xi)}\ar[rr]\ar[dd]&&X'\ar@{=}[dd]\ar[rr]&&o\\
&&&&&\\
0\ar[rr]&&M\ar[d]\ar[rr]&&Y'\ar[rr]\ar[d]&&X'\ar[rr]&&0\\
&&0&&0&&
}
$$
The short exact sequence $0 \to M \to Y' \to X' \to 0$ gives an $n$-FPD hull of $M$.  
Moreover, since $X_M$ and $Y_M$ share no common free summand via $i_M$, the homomorphism $R^{\oplus(m + \xi)} \to Y'$ is a minimal free cover of $Y'$.  
Therefore, we obtain the equalities
$$
\mu(Y^M)-\mu(X^M)=\mu(Y')-\mu(X')=m+\xi-m=\xi=\xi(n,M).
$$
\end{proof}

Auslander \cite{AusM} defined an invariant called the \emph{index} of a Gorenstein local ring.  
Let $(R, \mathfrak{m}, k)$ be a Gorenstein local ring.  
The \emph{index} of $R$ is defined as the infimum of all integers $n \ge 0$ such that $\delta_R(R /{\mathfrak{m}^n}) \ne 0$.  
Many interesting and deep studies have been done on this invariant; see, for example,  
\cite{AusM, ABIM, Ding92, Ding93, Ding94, Ding942, HS, Kob, Mart96I, Puthen, Shida, Yos93, Yos96}.
For the index of the $\xi$-invariant introduced by Martsinkovsky over an arbitrary commutative noetherian local ring , the well-definedness is stated in \cite[Proposition~4.3]{Mart96I}.  
Although the proposition below is essentially shown in the proof there, we shall present a more detailed statement with a proof based on stable module theory.
Let $\H_\m^i(-)$ denote the $i$-th \emph{local cohomology functor}, that is, the $i$-th right derived functor of the $\m$-torsion functor $\Gamma_\m(-)$, defined by $\Gamma_\m(M) = \{x \in M \mid \m^r x = 0 \text{ for some } r > 0\}$ for an $R$-module $M$.  
More explicitly, we have $\H_\m^i(-) = \varinjlim_{n > 0} \Ext^i_R(R/\m^n, -)$.

\begin{prop}\label{index}
Assume that $R$ has Krull dimension $d$.  
Then there exists an integer $n > 0$ such that $\xi_R(d, R / \m^n) \ne 0$.
\end{prop}

\begin{lem}\cite[Lemma 1.35]{AB}\label{ABex}
Let $\Lambda$ be a noetherian ring.
Let $0\to X\to Y\to Z\to 0$ be a short exact sequence in $\mod\Lambda$. Then there exists a long exact sequence
$$
\lhom_\Lambda(-,X)\to\lhom_\Lambda(-,Y)\to\lhom_\Lambda(-,Z)\to\Ext^1(-,X)\to\Ext^1_\Lambda(-,Y)\to\Ext^1_\Lambda(-,Z)\to\cdots
$$
of functors.
Hence, for $M\in\mod\Lambda$, there is an injection $\lhom_\Lambda(-,M)\hookrightarrow\Ext^1_\Lambda(-,\syz M)$ of functors.
\end{lem}

\begin{proof}[Proof of Proposition \ref{index}]
Consider the sequence of natural surjections
$$
\cdots \twoheadrightarrow R/{\m^n} \overset{\pi_{n,n-1}}{\twoheadrightarrow} R/{\m^{n-1}} \twoheadrightarrow \cdots \twoheadrightarrow R/{\m^2} \overset{\pi_{2,1}}{\twoheadrightarrow} R/{\m}.
$$
It follows from Lemma~\ref{ABex} that we obtain the following commutative diagram.
$$
\xymatrix@R-1pc@C-1pc{
\lhom(\syz^dk,\syz^dk)\ar@{^{(}->}[d]\ar[r]&\cdots\ar[r]&\lhom(\syz^d (R/{\m^{n-1}}),\syz^dk)\ar@{^{(}->}[d]\ar[r]&\lhom(\syz^d (R/{\m^{n}}),\syz^dk)\ar@{^{(}->}[d]\ar[r]&\cdots\\
\Ext^1(\syz^dk,\syz^{d+1}k)\ar@{=}[d]\ar[r]&\cdots\ar[r]&\Ext^1(\syz^d (R/{\m^{n-1}}),\syz^{d+1}k)\ar@{=}[d]\ar[r]&\Ext^1(\syz^d (R/{\m^{n}}),\syz^{d+1}k)\ar@{=}[d]\ar[r]&\cdots\\
\Ext^{d+1}(k,\syz^{d+1}k)\ar[r]&\cdots\ar[r]&\Ext^{d+1}(R/{\m^{n-1}},\syz^{d+1}k)\ar[r]&\Ext^{d+1}(R/{\m^n},\syz^{d+1}k)\ar[r]&\cdots
}
$$
Taking the direct limit, we obtain from Grothendieck's vanishing theorem \cite[Theorem~3.5.7]{BH} that  
$$
\varinjlim_{n > 0} \lhom(\syz^d(R / \m^n), \syz^d k) \hookrightarrow \varinjlim_{n > 0} \Ext^{d+1}(R / \m^n, \syz^{d+1} k) \cong \H^{d+1}_{\m}(\syz^{d+1} k) = 0.
$$
In particular, under the natural homomorphism $\lhom(\syz^d k, \syz^d k) \to \varinjlim_{n > 0} \lhom(\syz^d(R / \m^n), \syz^d k)$, the identity map $\underline{1_{\syz^d k}}$ is mapped to zero.
Then there exists a sufficiently large integer $n > 0$ such that for the composite $\pi_n = \pi_{n,n-1} \circ \cdots \circ \pi_{2,1} : R /{\m^n} \twoheadrightarrow k$, the morphism $\underline{\syz^d \pi_n} : \syz^d(R / 
{\m^n}) \to \syz^d k$ is zero in $\lmod R$.  
Then $\pi_n \in \V_d(R /{\m^n})$, which means that $\xi(d, R /{\m^n}) = 1$.
\end{proof}

We compare the $n$-th approximated $\xi$-invariant over $R$ with that over $R / xR$, where $x \in R$ is an $R$-regular element.
The following inequality holds true.

\begin{prop}\label{reg}
Let $M$ be a finitely generated $R$-module and $n > 0$ an integer.
Let $x$ be an element of $R$ which is regular on both $R$ and $M$.
Then the following inequalities hold:
$$
\xi_{R/{xR}}(n-1, M/{xM}) \le \xi_R(n, M) \le \xi_{R/{xR}}(n, M/{xM}).
$$
\end{prop}

We shall need the following well-known lemmas to compare syzygies over $R$ and those over $R / xR$.

\begin{lem}\label{reglem}
Let $n \ge 0$ be an integer, and let $x \in \m$.  
Then there is a natural transformation  
$$(\syz^n_R(-)) \otimes_R R/{xR} \to \syz^n_{R/xR}((-)\otimes_R R/{xR})$$ 
of functors from $\lmod R$ to $\lmod(R/{xR})$.  
Moreover, if $M$ is a finitely generated $R$-module and $x$ is regular on both $R$ and $M$, then the morphism  
$\syz^n_RM/{x\syz^n_RM} \to \syz^n_{R/xR}(M/{xM})$ 
is an isomorphism in $\lmod(R/{xR})$.
\end{lem}

\begin{proof}
Let $M$ be a finitely generated $R$-module.
From the commutative diagram with exact rows
$$
\xymatrix@R-1pc@C-1pc{
0\ar[r]&\syz_RM\ar[r]\ar[d]^-{x}&F\ar[r]\ar[d]^-{x}&M\ar[r]\ar[d]^-{x}&0\\
0\ar[r]&\syz_RM\ar[r] &F\ar[r] &M\ar[r]&0, 
}
$$
where $F$ is a free cover of $M$, we obtain an exact sequence $\syz_R/{x\syz_RM}\to F/{xF}\to M/{xM}\to0$, and this induces a natural transformation $\syz_R(-)/{x\syz_R(-)}\to\syz_{R/xR}((-)/{x(-)})$.
For an integer $n>1$, a natural transformation $\syz_R^n(-)/{x\syz^n_R(-)}\to \syz^n_{R/xR}((-)/{x(-)})$ is defined, inductively.
Indeed, assume that a natural transformation $\syz_R^{n-1}(-)/{x\syz^{n-1}_R(-)}\to \syz^{n-1}_{R/xR}((-)/{x(-)})$ is defined.
Then, by the composition
\begin{align*}
&\syz^n_R(-)/{x\syz^n_R(-)} 
= \syz_R(\syz^{n-1}_R(-))/{x\syz_R(\syz^{n-1}_R(-))} \\
&\to \syz_{R/xR}(\syz^{n-1}_R(-)/{x\syz^{n-1}_R(-)}) 
\to \syz_{R/xR} \syz^{n-1}_{R/xR}((-)/{x(-)}) 
= \syz^n_{R/xR}((-)/{x(-)}),
\end{align*}
the desired natural transformation is defined.
Suppose that $x$ is $R$-regular and $M$-regular and take a finite projective resolution $P^\bullet$ of $M$.
Then, by considering the short exact sequence $0\to P^\bullet\xrightarrow{x}P^\bullet\to P^\bullet/{x P^\bullet}\to 0$ of complexes, we have short exact sequences $0\to\syz^i_RM\xrightarrow{x}\syz^i_RM\to\syz^i_{R/xR}(M/{xM})\to0$ for all $i\ge0$.
The desired isomorphism is obtained.
\end{proof}

\begin{lem}\label{syzsyz}
Let $n \ge 1$ be an integer, and let $x \in \m$ be an $R$-regular element.  
Then the syzygy functor $\syz_R^n : \lmod R \to \lmod R$ induces a functor $\syz_R^n : \lmod(R/xR) \to \lmod R$.  
Moreover, for any integer $m \ge 0$, there exists a natural isomorphism of functors 
$\syz_R^{m+1}(-) \xrightarrow{\cong} \syz_R \circ \syz^m_{R/xR}(-)$ from $\lmod(R/xR)$ to $\lmod R$.
\end{lem}

\begin{proof}
Applying the $n$-th syzygy functor $\syz_R^n$ to any finitely generated $R/xR$-module yields a projective $R$-module, so $\syz^n_R : \lmod(R/xR) \to \lmod R$ is indeed induced.  

Let $N$ be a finitely generated projective $R/xR$-module.  
Take a finite free cover $R^{\oplus n} \twoheadrightarrow N$ over $R$.  
Then $(R/xR)^{\oplus n} \twoheadrightarrow N$ is also a finite free cover over $R/xR$.  
Hence, we obtain the following commutative diagram with exact rows and columns:
$$
\xymatrix@R-1pc@C-1pc{
&0\ar[d] & 0\ar[d] &&\\
&R^{\oplus n} \ar@{=}[r] \ar[d] & R^{\oplus n} \ar[d] &&\\
0\ar[r] & \syz_R N \ar[r] \ar[d] & R^{\oplus n} \ar[r] \ar[d] & N \ar[r] \ar@{=}[d] & 0\\
0\ar[r] & \syz_{R/xR} N \ar[r] \ar[d] & (R/xR)^{\oplus n} \ar[r] \ar[d] & N \ar[r] & 0\\
&0 & 0 &&
}
$$
Applying $\syz_R$ to the first column, we obtain an isomorphism  
$\syz^2_R N \xrightarrow{\cong} \syz_R \syz_{R/xR}(N)$ in $\lmod R$,  
which induces a natural isomorphism of functors  
$\syz^2_R(-) \xrightarrow{\cong} \syz_R \syz_{R/xR}(-)$ from $\lmod(R/xR)$ to $\lmod R$.
For $m > 1$, we obtain the desired natural isomorphism  
$\syz_R^{m+1}(-) \xrightarrow{\cong} \syz_R \syz^m_{R/xR}(-)$ by induction.  
Indeed, suppose that we have a natural isomorphism  
$\syz_R^m(-) \xrightarrow{\cong} \syz_R \syz^{m-1}_{R/xR}(-)$.  
Then the composition of the following natural transformations of functors from $\lmod(R/xR)$ to $\lmod R$
$$
\syz^{m+1}_R(-)=\syz_R \syz^m_R(-)\xrightarrow{\cong}\syz_R \syz_R \syz^{m-1}_{R/xR}(-)=\syz^2_R \syz^{m-1}_{R/xR}(-)\xrightarrow{\cong}\syz_R \syz_R \syz^{m-1}_{R/xR}(-)=\syz_R \syz^m_{R/xR}(-)
$$
yields the desired natural isomorphism.
\end{proof}

\begin{proof}[Proof of Proposition \ref{reg}]
We first prove the inequality $\xi_R(n,M) \le \xi_{R/xR}(n, M/xM)$.  
To begin with, note that there exists an isomorphism  
$\Hom_R(M, k) \cong \Hom_{R/xR}(M/xM, k)$; $f \mapsto f/x f$.  
To prove the desired inequality, it suffices to show that if an $R$-homomorphism  
$f : M \to k$ satisfies $\underline{\syz^n_R f} = \underline{0}$ in $\lmod R$,  
then the morphism $\underline{\syz^n_{R/xR}(f/x f)}$ is also zero in $\lmod(R/xR)$.
By Lemma \ref{reglem}, the following diagram commutes in $\lmod(R/xR)$:
$$
\xymatrix@R-1pc@C-1pc{
{\syz^n_R M}/{x \syz^n_R M} \ar[dd]^-{\cong} 
  \ar[rrrrrr]^-{\underline{{\syz^n_R f}/x{\syz^n_R f}}} &&&&&&
{\syz^n_R k}/{x \syz^n_R k} \ar[dd] \\
&&&&&&\\
\syz^n_{R/xR}(M/xM) \ar[rrrrrr]^-{\underline{\syz^n_{R/xR}(f/x f)}} &&&&&&
\syz^n_{R/xR} (k)
}
$$
By assumption, the morphism $\underline{{\syz^n_R f}/x{\syz^n_R f}}$ is zero in $\lmod(R/xR)$, and so is $\underline{\syz^n_{R/xR}(f/x f)}$.  
Thus, the first inequality is proved.

Next, we prove the inequality $\xi_{R/xR}(n-1 , M/xM) \le \xi_R(n, M)$ holds true.  
That is, we need to show that if $\underline{\syz^{n-1}_{R/xR}(f/x f)} = \underline{0}$ in $\lmod(R/xR)$,  
then it follows that $\underline{\syz^n_R f} = \underline{0}$ in $\lmod R$.
By Lemma \ref{syzsyz}, the following commutative diagram commutes in $\lmod R$:
$$
\xymatrix@R-1pc@C-1pc{
\syz^n_R M \ar[rrrrrr]^-{\underline{\syz^n_R f}} \ar[dd]^-{\underline{\syz^n_R \pi}} &&&&&& \syz^n_R k \ar@{=}[dd] \\
&&&&&& \\
\syz^n_R (M/xM) \ar[rrrrrr]^-{\underline{\syz^n_R(f/x f)}} \ar[dd]^-{\cong} &&&&&& \syz^n_R k \ar[dd]^-{\cong} \\
&&&&&& \\
\syz_R \syz^{n-1}_{R/xR}(M/xM) \ar[rrrrrr]^-{\underline{\syz_R \syz^{n-1}_{R/xR}(f/x f)}} &&&&&& \syz_R \syz^{n-1}_{R/xR}(k),
}
$$
where $\pi : M \twoheadrightarrow M/xM$ is the natural surjection.
Again by Lemma \ref{syzsyz}, the assumption that $\underline{\syz^{n-1}_{R/xR}(f/x f)} = \underline{0}$ in $\lmod(R/xR)$ implies that  
$\underline{\syz_R \syz^{n-1}_{R/xR}(f/x f)}$ is the zero morphism in $\lmod R$.  
It follows from the above commutative diagram that $\underline{\syz^n_R f} = \underline{0}$ in $\lmod R$,  
and the proof is completed.
\end{proof}

We collect several important results that are recovered as special cases of our results.

\begin{cor}\label{cors}
Let $M$ be a finitely generated $R$-module.
\begin{enumerate}[\rm(1)]
   \item
   Suppose that $R$ is Gorenstein.
   Then the following hold.
   \begin{itemize}
      \item[(1-i)]{\rm\cite[Proposition 2.3]{Mart96I}}
      One has the equality $\delta_R(M)=\xi_R(M)$.
      \item[(1-ii)]{\rm\cite{AusM}}
      The equality $\delta_R(M)=0$ holds true if and only if any $R$-homomorphism $f:X\to Z$ of finitely generated $R$-modules with $\pd_R Z<\infty$ satisfies $f\otimes_R k=0$.
      \item[(1-iii)]{\rm\cite{AusM}}
      For any FPD hull $0\to M\to Y\to X\to 0$ of $M$, one has the equality $\delta_R(M)=\mu_R(Y)-\mu_R(X)$.
   \end{itemize}
   \item{\rm\cite[Theorem 8]{Mart96II}}
   Let $x$ be an element of $R$ which is regular on both $R$ and $M$.
   Then one has the equality $\xi_R(M)=\xi_{R/xR}(M/{xM})$.
   \item{\rm\cite[Theorem 12]{Mart96II}}
   Suppose that $\Ext^i_R(M,R)=0$ for all $i>0$.
   Then $\xi_R(M)$ is the rank of the largest free summand of $M$.
\end{enumerate}
\end{cor}

\begin{proof}
The statements (1-i) to (1-iii) follow from Theorem \ref{ABrank}, Theorem \ref{orginm}, and Proposition \ref{FPDxi}, respectively.  
The statement (2) follows from Proposition \ref{reg}, since we have
\[
\xi_{R/xR}(M/xM)
= \lim_{n \to \infty} \xi_{R/xR}(n-1, M/xM)
\le
\lim_{n \to \infty} \xi_R(n, M)
\le \lim_{n \to \infty} \xi_{R/xR}(n, M/xM)
= \xi_{R/xR}(M/xM)
\]
and the middle limit is none other than $\xi_R(M)$.
The statement (3) follows from Corollary \ref{symm}, which implies that $\xi_R(n, M) = \xi_R(0, M)$ for all integers $n \ge 0$.  
Indeed, from this equality we obtain that
$
\xi_R(M) = \lim_{n \to \infty} \xi_R(n, M) = \lim_{n \to \infty} \xi_R(0, M) = \xi_R(0, M).
$
\end{proof}

\begin{ac}
The author is grateful to his supervisor Ryo Takahashi for his many valuable suggestions.
\end{ac}

\end{document}